\documentclass[a4paper,fleqn]{article}

\usepackage[a4paper,left=1.5cm,right=1.5cm,top=2cm,bottom=2cm]{geometry}
\usepackage[T1]{fontenc}
\usepackage[utf8]{inputenc} % remove if using LuaLaTeX/XeLaTeX
\usepackage[english]{babel}

\usepackage{times} % keep this only if you want minimal changes
\usepackage{mathtools,amssymb,amsthm}
\usepackage{cite}

\usepackage{graphicx}
\usepackage{enumitem}
\usepackage{xcolor}
\usepackage{tikz,pgfplots}
\usepackage{todonotes}

\usepackage{xspace}
\usepackage{authblk}

\usepackage{hyperref}
\usepackage{cleveref}

\DeclareMathOperator*{\esssup}{ess\,sup}

\newcommand{\R}{\mathbb{R}}
\newcommand{\N}{\mathbb{N}}
\renewcommand{\d}{{\mathrm d}}
\renewcommand{\div}{\mathrm{div}}
\newcommand{\defeq}{\vcentcolon=}
\newcommand{\eqdef}{=\vcentcolon}

\makeatletter
\DeclareRobustCommand\onedot{\futurelet\@let@token\@onedot}
\def\@onedot{\ifx\@let@token.\else.\null\fi\xspace}

\def\eg{\emph{e.g}\onedot} 
\def\ie{\emph{i.e}\onedot} 
\def\cf{\emph{cf}\onedot}

\makeatother

\newcommand{\info}[1]{\todo[inline, color = yellow, disable]{#1}}
\newcommand{\notinclude}[1]{}

\newtheorem{theorem}{Theorem}[section]
\newtheorem{lemma}[theorem]{Lemma}
\newtheorem{proposition}[theorem]{Proposition}
\newtheorem{corollary}[theorem]{Corollary}
\newtheorem{assumption}[theorem]{Assumption}
\theoremstyle{definition}

\newtheorem{remark}[theorem]{Remark}
\title{On MAP estimates and source conditions for drift identification in SDEs}

\author[1]{Daniel Tenbrinck}
\author[2]{Nikolas Uesseler\footnote{Corresponding author:  nikolas.uesseler@uni-muenster.de,
	    phone +49\,251\,83-30410,
	    fax +49\,251\,83-32729}}
\author[3]{Philipp Wacker}
\author[2]{Benedikt Wirth}
\affil[1]{Department of Data Science (DDS), Friedrich-Alexander-Universit\"at Erlangen-N\"urnberg, N\"urnberger Str.\ 74, 91058 Erlangen}
\affil[2]{Applied Mathematics M\"unster, University of M\"unster, Einsteinstr.\ 62, 48149 M\"unster}
\affil[3]{Department of Mathematics and Statistics, University of Canterbury, Private Bag 4800, Christchurch 8140, New Zealand}

\date{}

\begin{document}
	\maketitle
	
\begin{abstract}
We consider the inverse problem of identifying the drift in an SDE from $n$ observations of its solution at $M+1$ distinct time points.
We derive a corresponding MAP estimate, we prove differentiability properties as well as a so-called tangential cone condition for the forward operator,
and we review the existing theory for related problems, which under a slightly stronger tangential cone condition would additionally yield convergence rates for the MAP estimate as $n\to\infty$.
Numerical simulations in 1D indicate that such convergence rates indeed hold true.
\end{abstract}
%% maketitle must follow the abstract.
\maketitle                   % Produces the title.

\section{Introduction}
We are considering the following problem.
At fixed time points $0=t_0< t_1< \ldots< t_M = T$ we observe the position of $n$ distinguishable particles which stochastically move around in some bounded smooth domain $\Omega\subset\R^d$.
From this observation we aim to estimate the drift $\mu:\Omega\to\R^d$ in the stochastic differential equation (SDE)
\begin{equation}\label{eqn:SDE}
\d X_t=\mu(X_t)\d t+\sigma\d W_t
\end{equation}
that governs the motion of each particle ($W_t$ denotes the Wiener process with reflection at $\partial\Omega$ and $\sigma>0$ a known coefficient;
furthermore we assume the normal component $\mu\cdot\nu$ of the drift to vanish on the boundary $\partial\Omega$).
Our aim is to develop and analyse a corresponding Bayesian maximum a posteriori (MAP) estimate.

An example application for which this inverse problem is relevant is the probing of tissue properties in zebrafish embryos:
During embryonic development, so-called primordial germ cells (PGCs) migrate from the location of their differentiation to the site where the gonads develop.
However, if directional migration is abolished in those cells by genetic modification, they migrate randomly within the embryo \cite{Gross-Thebing}.
The drift $\mu$ then indicates areas of in-/decreased cell attraction within the embryo (\cf \cref{fig:embryo}).
\begin{figure}[h]
    \centering
    \includegraphics[width=0.475\textwidth,trim=0 30 0 50,clip]{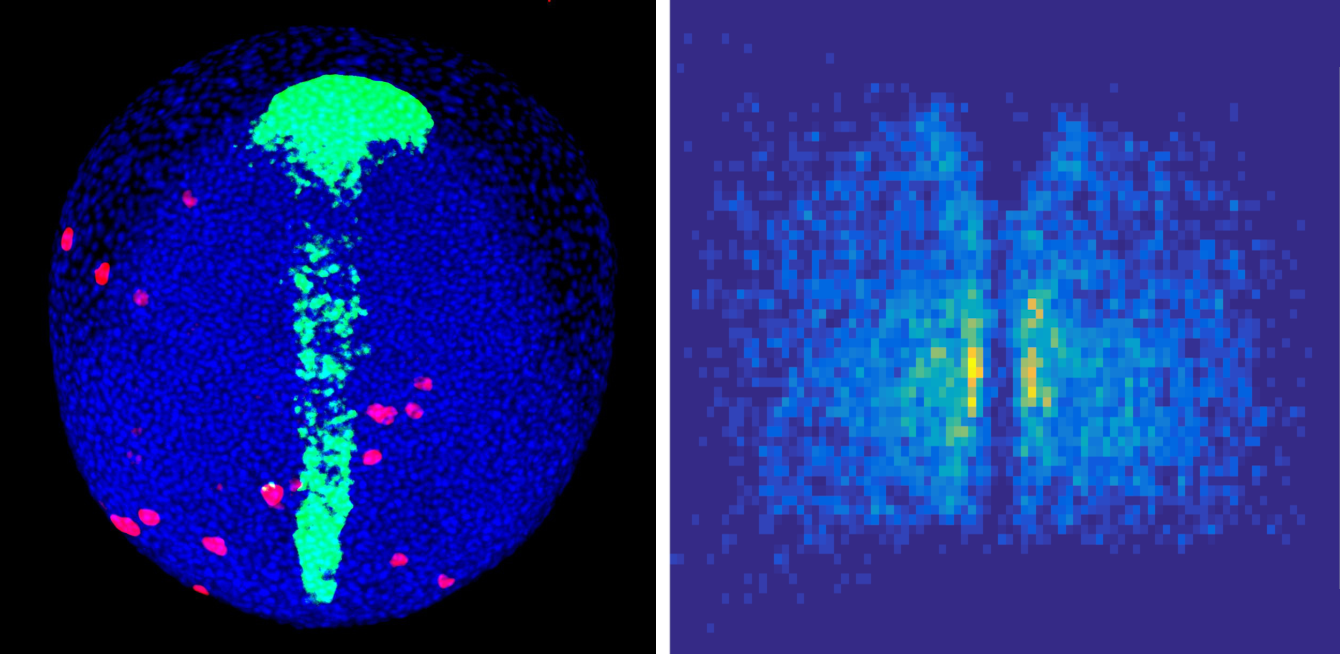}
    \caption{Left: Epifluorescence microscopy data of a zebrafish embryo (PGCs in red). Right: PGC distribution after registration of 934 embryos reveals anatomical structures as barrier for cells \cite{Gross-Thebing}.}
    \label{fig:embryo}
\end{figure}

As usual, our MAP estimate will be the minimizer of the (slightly modified) log-posterior
\begin{equation*}
\mu\mapsto S_\tau(F(\mu),G^n)+\alpha\|\mu\|_{H^r}^2,
\end{equation*}
where $G^n$ is the empirical measure of the observed particles, the forward operator $F$ yields the probability density of particle observations under a given drift $\mu$,
$S$ is the cross-entropy and $S_\tau$ a shifted version, and $\alpha>0$ a regularization parameter.
The parameter identification being a nonlinear inverse problem, we are interested in
whether or under what conditions one can obtain rates for how fast the minimizer of the above functional converges to the ground truth $\mu^\dagger$
as the particle number $n$ increases to infinity and the regularization weight $\alpha$ is decreased correspondingly.
To this end we extensively review the existing theory and perform first steps to applying this theory to our setting:
\begin{itemize}
\item
We give an overview of the historical development of source and nonlinearity conditions for nonlinear inverse problems,
culminating in variational source conditions that are particularly pertinent to our framework (\cref{sec:history,sec:generalizedFidelity}).
\item
We next present the forward operator relevant to our study and derive the corresponding log-likelihood for the associated inverse problem (\cref{sec:forwardOp}).
\item
The following sections review theoretical results from \cite{Dunker-Hohage} (which considered the stationary version of our inverse problem) on stochastic convergence rates in the context of stochastic inverse problems,
highlighting how these (albeit non-explicit) rates may be established using the variational source conditions previously discussed and tailored to our specific setup (\cref{sec:probabilisticRate,sec:reduction}).
\item
Within this framework, the variational source conditions can be reduced to a tangential cone condition.
In the penultimate section, we prove a weaker form of this condition; although it does not suffice to guarantee convergence rates, it nevertheless provides supporting evidence that such rates may hold (\cref{sec:tangentialCone}).
\item
We conclude by presenting numerical experiments that illustrate and support our theoretical findings (\cref{sec:numerics}).
\end{itemize}

%relation between the SDE and the Fokker Planck equation, heuristic explanation of the result of dunker and hohage in the stationary case, heuristic explantion of the new result, structure of the paper and notation.
% section 2: general tikhonov functional, general operator
% section 3: presenting our tikhonov functional and our operator
% section 4: Convergence rate for specific tikhonov functional, general operator whose images are probability kernels, and under the assumption of the variational source condition.
%section 5: Convergence rate for specific tikhonov functional and specific operator
%%---------------------------------------------------------------------------------------------------
\section{Historical note on convergence rates for nonlinear inverse problems}\label{sec:history}
In their 1988 paper \cite{Seidman-Vogel}, Seidmann and Vogel extended the theory of ill-posed problems to the case of nonlinear operators $F:X\to Y$.
As an approximate solution to $F(x)=y^\delta$ for some noisy measurement $y^\delta$ they consider the minimizer $x_\alpha^\delta$ of the Tikhonov functional
\begin{equation*}
	T_{\alpha}^{y^\delta}(x) \defeq \|F(x)-y^\delta\|^2 + \alpha  \|x - x^\star\|^2  \quad \text{for } x \in X\text{ and some fixed }x^{\star}\in X
\end{equation*}
and study its existence, stability, and convergence to the ground truth $x^\dagger=F^{-1}(y^\dagger)$ for $y^\delta\to y^\dagger$ and an appropriate choice of the regularization parameter $\alpha>0$.
%These are results of great theoretical importance but in practical application it is as important to have information about the speed of the convergence.
In 1989, Engl, Kunisch, and Neubauer then published the first result on corresponding convergence rates \cite{Engl-Kunisch-Neubauer},
for which one needs to impose so-called source conditions on the ground truth.
In this section we give a historical account on the development of these ideas to the later employed so-called variational source condition.

The convergence of Tikhonov regularization can be arbitrarily slow without a source condition, the concept of which we briefly motivate: Assume that $x^\dagger$ is an $x^\star$-minimum norm solution, \ie a solution of
\begin{equation}\label{eqn:minimumNormSolution}
	\text{minimize}\qquad  \tfrac{1}{2}\|x^\star - x \|^2 \qquad  \text{subject to}\qquad F(x) = y^\dagger.
\end{equation}
Naturally we can consider the Lagrangian
\begin{equation*} \label{eqn:ClassicalLagrangian}
	L(\omega, x) = \tfrac{1}{2}\|x^\star - x \|^2 + \langle\omega, F(x) - y^\dagger\rangle
\quad\text{for }x\in X,\ \omega \in Y^*.
\end{equation*}
% for $\langle\cdot,\cdot\rangle$ the respective dual pairing.
Let us assume that $F$ is Fréchet-differentiable and that strong duality holds, \ie if $x^\dagger \in X$ minimizes \eqref{eqn:minimumNormSolution} then there exists $\hat{\omega} \in Y^*$ such that $(\hat{\omega}, x^\dagger)$ is a saddle point of $L$. In particular, under this assumption it holds
\begin{equation*}
	\tfrac{\partial}{\partial x} L(\hat{\omega},x)|_{x^\dagger} = 0
\qquad\text{or equivalently}\qquad
	x^\star - x^\dagger = F'(x^\dagger)^\# \hat{\omega},
\end{equation*}
understood as an equality in $X^*$.
For instance, in the case of $x^\star = 0$ and linear compact $F$ between Hilbert spaces $X,Y$ we arrive at the well-known source condition $x^\dagger = F^\# \hat\omega = (F^\#F)^\frac{1}{2}p$ for some $p\in X$, which is to be understood as imposed regularity of $x^\dagger$ measured in terms of $F$. Thus, this regularity assumption is generalized by the assumption of strong duality, which, for the general nonlinear problem, indeed allows to achieve convergence rates. Let us summarize the result from \cite{Engl-Kunisch-Neubauer}.

\begin{assumption}[Conditions for convergence rates I, \cite{Engl-Kunisch-Neubauer}] \label{assumption Neubauer}
\begin{enumerate}
\item
	Let $X,Y$ be Hilbert spaces and $F:X \to Y$ Fréchet-differentiable.% with convex domain $\mathcal{D}(F)$.
\item
	Let $x^\dagger$ be an $x^\star$-minimum norm solution that fulfils the following \emph{source condition}: There exists $\omega \in Y$ satisfying
	\begin{equation*}
		x^\dagger - x^\star = F'(x^\dagger)^\# \omega.
	\end{equation*}
\item
	Let $F'$ fulfil the following \emph{Lipschitz condition}: There exists $L<1/\|\omega\|$ with
	\begin{equation*}
		\| F'(x^\dagger) - F'(x) \| \leq L \| x^\dagger - x \| \qquad\text{ for all } x \in X.
	\end{equation*}
% 	The Lagrange multiplier $\omega$ is small enough, namely
% 	\begin{equation*}
% 		L \| \omega \| < 1.
% 	\end{equation*}
\end{enumerate}
\end{assumption}

\begin{theorem}[Convergence rate I, \protect{\cite[Thm.\,2.4]{Engl-Kunisch-Neubauer}}] \label{Convergence rates Neubauer}
	Let assumption \ref{assumption Neubauer} hold and %let $y^\delta \in Y$ be measurements of $y^\dagger$ with
  $\| y^\delta - y^\dagger \| \leq \delta$.
  %Let as before $x_\alpha^\delta$ be the minimizer of the Tikhonov functional $T_\alpha^\delta$.
	Then for the choice $\alpha \sim \delta$ it holds 
	$
		\|x_\alpha^\delta - x^\dagger \| \lesssim \sqrt{\delta}.
	$
\end{theorem}

The next step, made by Burger and Osher in 2004 \cite{Burger-Osher}, was to replace the squared Hilbert space norm $\|x- x^\star \|^2$ by a general convex penalty functional $J$, in which case one can still estimate the reconstruction error in terms of the \emph{Bregman distance}
\begin{equation*}
\mathcal D_J^\xi(x,x^\dagger)=J(x)-J(x^\dagger)-\langle\xi,x-x^\dagger\rangle\geq0\qquad\text{for }\xi\in\partial J(x^\dagger).
% ,\qquad\mathcal D_J(x,x^\dagger)=\{\mathcal D_J^\xi(x,x^\dagger)\,|\,\xi\in\partial J(x^\dagger)\}.
\end{equation*}
% In their paper they rewrite the source condition we discussed in last section in terms of subdifferentiability.
As before, let $x^\dagger$ be a $J$-minimizing solution of $F(x)=y^\dagger$ and consider the corresponding Lagrangian
\begin{align*}
	L(\omega, x) = J(x) + \langle \omega, F(x) - y^\dagger\rangle.
\end{align*}
%Using the fact that for a convex proper function $f$ minimality at $x$ is equivalent to $0 \in \partial f(x)$ 
Now the strong duality assumption is equivalent to the existence of a Lagrange multiplier $\omega$ such that $(\omega, x^\dagger)$ is a saddle point of the Lagrangian, which with convexity of $J$ and differentiability of $F$ implies the (generalized) source condition
\begin{align*}
	F'(x^\dagger)^\# \omega \in \partial J(x^\dagger).
\end{align*}
Again, a condition controlling the nonlinearity of operator $F$ is needed in addition. %Let us formulate the assumptions and the result.
\begin{assumption}[Conditions for convergence rates II, \cite{Burger-Osher}]\label{assumption Burger}
\begin{enumerate}
\item
	Let $X$ be a Banach space carrying also a potentially weaker topology $\tau_X$, $Y$ a Hilbert space, and $F:X \to Y$ sequentially continuous w.r.t.\ $\tau_X$ and Fréchet-differentiable. % on the (nonempty) interior of its convex domain $\mathcal{D}(F)$.
  Let $J$ be convex and sequentially lower semi-continuous w.r.t.\ $\tau_X$ with $\tau_X$-sequentially precompact sublevel sets.
\item
  Let $x^\dagger$ be a $J$-minimizing solution that fulfils the following \emph{source condition}: There exists $\omega \in Y$ satisfying
	\begin{align*}
		\xi\defeq F'(x^\dagger)^\# \omega \in \partial J(x^\dagger).
	\end{align*}
\item
	Let $F$ fulfil the following \emph{nonlinearity condition}: There exists $\eta > 0$ %and $r>0$ with
	\begin{align*}
		\langle F(x) - F(x^\dagger) -F'(x^\dagger)(x-x^\dagger), \omega\rangle \leq \eta \| F(x^\dagger) - F(x) \| \|\omega \| && \text{ for all } x \in X.% \mathcal{B}_r(x^\dagger).
	\end{align*}
\end{enumerate}
\end{assumption}

\begin{theorem}[Convergence rate II, \protect{\cite[Sec.\,3.3]{Burger-Osher}}]\label{convergence rates Burger}
	Let assumption \ref{assumption Burger} hold and $\| y^\delta - y^\dagger \| \leq \delta$.
	Then for the choice $\alpha \sim \delta$  it holds $\mathcal{D}_J^\xi(x_\alpha^\delta,x^\dagger) \lesssim \delta$.
\end{theorem}

In 2006, Resmerita and Scherzer \cite{Resmerita-Scherzer} allowed also $Y$ to be Banach and replaced the nonlinearity control on $F$ by a Bregman distance. 
\begin{assumption}[Conditions for convergence rates III, \cite{Resmerita-Scherzer}] \label{assumption Resmerita}
\begin{enumerate}
\item
	Let $X,Y$ be Banach spaces, both carrying potentially weaker topologies $\tau_X$ and $\tau_Y$, and $F:X \to Y$ sequentially continuous w.r.t.\ $\tau_X$ and $\tau_Y$ and Fréchet-differentiable. %on the (nonempty) interior of its convex domain $\mathcal{D}(F)$.
  Let $J$ be convex and sequentially lower semi-continuous w.r.t.\ $\tau_X$ with $\tau_X$-sequentially precompact sublevel sets, and let the norm on $Y$ be sequentially lower semi-continuous w.r.t.\ $\tau_Y$.
\item
  Let $x^\dagger$ be a $J$-minimizing solution that fulfils the following \emph{source condition}: There exists $\omega \in Y$ satisfying
	\begin{align*}
		\xi\defeq F'(x^\dagger)^\# \omega \in \partial J(x^\dagger).
	\end{align*}
\item
	Let $F$ fulfil the following \emph{nonlinearity condition}: There exists $\eta<1/\|\omega\|$ %and $r>0$ with
	\begin{equation*}
		\|F(x) - F(x^\dagger) - F'(x^\dagger)(x - x^\dagger) \| \leq \eta\mathcal{D}_J^\xi(x, x^\dagger) \qquad \text{ for all } x\in X.% \in \mathcal{B}_r(x^\dagger).
	\end{equation*}
\end{enumerate}
\end{assumption}

\begin{theorem}[Convergence rate III, \protect{\cite[Thm.\,3.2]{Resmerita-Scherzer}}] \label{convergence rates Resmerita}
	Let assumption \ref{assumption Resmerita} hold and $\| y^\delta - y^\dagger \| \leq \delta$.
	Then for the choice $\alpha \sim \delta$ it holds
	$
		\| F(x_\alpha^\delta) - F(x^\dagger) \| \lesssim \delta
	$
	and
	$
		\mathcal{D}_J^\xi(x_\alpha^\delta, x^\dagger) \lesssim \delta.
	$
\end{theorem}

To this point the nonlinearity conditions on $F$ and the source conditions on $x^\dagger$ require differentiability of $F$ at $x^\dagger$.
In 2007, Hofmann, Kaltenbacher, Pöschl, and Scherzer managed to remove this restriction via a reformulation as a variational inequality \cite{Hofmann-Kaltenbacher-Poschl-Scherzer},
arriving at a so-called \emph{variational source condition}. %(which we later relate to the previous conditions).
They also slightly generalize the Tikhonov functional to
\begin{equation} \label{eqn:TikhonovFunctionalPNorm}
	T_{\alpha}^{\delta}(x) \defeq \|F(x)-y^\delta\|^p + \alpha J(x)\qquad \text{ for } p>1.
\end{equation}

\begin{assumption}[Conditions for convergence rates IV, \cite{Hofmann-Kaltenbacher-Poschl-Scherzer}] \label{assumption Hofmann}
\begin{enumerate}
\item
	Let assumption \ref{assumption Resmerita} (1) hold, except Fréchet differentiability of $F$ is no longer required.
\item
	Let $x^\dagger$ be a $J$-minimizing solution, and let there exist $\xi \in \partial J(x^\dagger)$, $\beta_1 \in[0,1)$ and $\beta_2\geq0$ such that
	\begin{equation*}
		-\langle\xi, x - x^\dagger\rangle \leq \beta_1  \mathcal{D}_J^\xi(x, x^\dagger) + \beta_2 \| F(x) - F(x^\dagger) \| \qquad \text{ for all } x \in X.% \mathcal{B}_r(x^\dagger).
	\end{equation*}
\end{enumerate}
\end{assumption}

\begin{remark}[Source and nonlinearity condition imply variational source condition]
  Each of the previous assumptions implies assumption \ref{assumption Hofmann}, so the latter is the weakest.
  Indeed, assumption \ref{assumption Resmerita} is already weaker than assumption \ref{assumption Neubauer}, and given assumption \ref{assumption Resmerita} or \ref{assumption Burger} we can argue as follows:
  Pick $\xi=F'(x^\dagger)^\# \omega$, then
  \begin{equation*}
		\langle\xi, x - x^\dagger\rangle
		= \langle\omega,F'(x^\dagger) (x - x^\dagger) + F(x^\dagger) - F(x)\rangle + \langle\omega, F(x) - F(x^\dagger)\rangle,
  \end{equation*}
  which under assumption \ref{assumption Burger} results in
	%In the special case where assumption \ref{assumption Burger} and in particular differentiability of $F$ is given, the following holds
	\begin{equation*}
		-\langle\xi, x - x^\dagger\rangle
		%&= \langle F'(x^\dagger)^\# \omega, x - x^\dagger\rangle \\
		%&= \langle\omega,F'(x^\dagger) (x - x^\dagger)\rangle \\
		%&= \langle\omega,F'(x^\dagger) (x - x^\dagger) + F(x^\dagger) - F(x) - F(x^\dagger) + F(x)\rangle \\
		%&= \langle\omega,F'(x^\dagger) (x - x^\dagger) + F(x^\dagger) - F(x)\rangle + \langle\omega, F(x) - F(x^\dagger)\rangle \\
		\leq \eta \|\omega \| \| F(x^\dagger) - F(x) \|  + \| \omega \| \| F(x) - F(x^\dagger) \|
		= \underbrace{(1 + \eta) \| \omega \|}_{\eqdef \beta_2} \| F(x^\dagger) - F(x) \|
	\end{equation*}
	and under assumption \ref{assumption Resmerita} results in
	\begin{equation*}
		-\langle\xi, x - x^\dagger\rangle
		%&= \langle F'(x^\dagger)^\# \omega, x - x^\dagger\rangle \\
		%&= \langle\omega,F'(x^\dagger) (x - x^\dagger)\rangle \\
		%&= \langle\omega,F'(x^\dagger) (x - x^\dagger) + F(x^\dagger) - F(x) - F(x^\dagger) + F(x)\rangle \\
		%&= \langle\omega,F'(x^\dagger) (x - x^\dagger) + F(x^\dagger) - F(x)\rangle + \langle\omega, F(x) - F(x^\dagger)\rangle \\
		\leq \| \omega \| \| F'(x^\dagger) (x - x^\dagger) + F(x^\dagger) - F(x) \| + \| \omega \| \|  F(x) - F(x^\dagger)\|
		\leq \underbrace{\eta \| \omega \|}_{\eqdef \beta_1}  \mathcal{D}_J^\xi(x, x^\dagger) + \underbrace{\| \omega \|}_{\eqdef \beta_2} \|  F(x) - F(x^\dagger)\|.
	\end{equation*}
	% which is assumption \ref{assumption Hofmann} again. \\
	% Finally recall that assumption \ref{assumption Neubauer} could already be replaced by assumption \ref{assumption Resmerita}. \\
	% All together we see that assumption \ref{assumption Hofmann} covers all other conditions we had studied in their respective context. In other words the following theorem can be applied in all the former situations.
\end{remark}

\begin{theorem}[Convergence rate IV, \protect{\cite[Thm.\,4.4]{Hofmann-Kaltenbacher-Poschl-Scherzer}}] \label{convergence rates Hofmann}
	Let assumption \ref{assumption Hofmann} hold and $\| y^\delta - y^\dagger \| \leq \delta$. %Let $x_\alpha^\delta$ be a minimizer of the Tikhonov functional $T_\alpha^\delta$ in \eqref{formula Tikhonov functional with p}, where $p > 1$.
	Then for the choice $\alpha \sim \delta^{p-1}$ it holds 
	$
		\| F(x_\alpha^\delta) - F(x^\dagger) \| \lesssim \delta
	$
	and
	$
		\mathcal{D}_J^\xi(x_\alpha^\delta, x^\dagger) \lesssim \delta.
	$
\end{theorem}

In his 2009 joint work with Yamamoto \cite{Hofmann-Yamamoto}, Hofmann himself points out a slight generalization of theorem \ref{convergence rates Hofmann}
in which the only change is a relaxation of the nonlinearity condition by some exponent $\kappa$.
 
\begin{assumption}[Conditions for convergence rates V, \cite{Hofmann-Yamamoto}]\label{assumption Yamamoto}
\begin{enumerate}
\item
	Let assumption \ref{assumption Hofmann} (1) hold.
\item
	Let $x^\dagger$ be a $J$-minimizing solution, and let there exist $\xi \in \partial J(x^\dagger)$, $\beta_1 \in[0,1)$, $\beta_2\geq0$, and $\kappa\in(0,1]$ such that
	\begin{equation*}
		-\langle\xi, x - x^\dagger\rangle \leq \beta_1  \mathcal{D}_J^\xi(x, x^\dagger) + \beta_2 \| F(x) - F(x^\dagger) \|^\kappa \qquad \text{ for all } x \in X.% \mathcal{B}_r(x^\dagger).
	\end{equation*}
\end{enumerate}
\end{assumption}
% This assumption only differs from assumption \ref{assumption Hofmann} in the exponent $\kappa$ in the right-hand side of the inequality. Now the corresponding result, whose proof is analogous to the one of theorem \eqref{convergence rates Hofmann} and hence also based on the Young inequality.
 
\begin{theorem}[Convergence rate V, \protect{\cite[Thm.\,3.3]{Hofmann-Yamamoto}}]
Let assumption \ref{assumption Yamamoto} hold and $\| y^\delta - y^\dagger \| \leq \delta$. %Let $x_\alpha^\delta$ be a minimizer of the Tikhonov functional $T_\alpha^\delta$ in \eqref{eqn:TikhonovFunctionalPNorm}, where $p > 1$.
Then for the choice $\alpha \sim \delta^{p-\kappa}$ it holds
$
  \mathcal{D}_J^\xi(x_\alpha^\delta, x^\dagger) \sim \delta^\kappa.
$
\end{theorem}
 
This result suggests that the variational inequality as well as the Tikhonov functional, which both simply use powers of the norm on Banach space $Y$, could be generalized and still provide rates of convergence. In fact, the same year Hofmann and Bo\textcommabelow{t} consider the Tikhonov functional \cite{Hofmann-Bot}
\begin{align*}
  T_{\alpha}^{\delta}(x) \defeq \psi(\|F(x)-y^\delta\|) + \alpha J(x).
\end{align*}

\begin{assumption}[Conditions for convergence rates VI, \cite{Hofmann-Bot}]\label{assumption Bot}
\begin{enumerate}
\item
Let assumption \ref{assumption Resmerita} (1) hold and $J$ be Gateaux-differentiable in $x^\dagger$.
\item
Let $\psi,\phi$ be twice differentiable index functions, \ie strictly increasing functions on $[0,\infty)$ with $\psi(0) = 0 = \phi(0)$. Function $\psi$ shall be strictly convex and $\phi$ concave.
\item
Let $x^\dagger$ be a $J$-minimizing solution, and let there exist $\xi \in \partial J(x^\dagger)$, $\beta_1 \in[0,1)$, $\beta_2\geq0$ such that
\begin{equation*}
  -\langle\xi, x - x^\dagger\rangle \leq \beta_1  \mathcal{D}_J^\xi(x, x^\dagger) + \beta_2 \phi(\| F(x) - F(x^\dagger) \|) \qquad \text{ for all } x \in X.% \mathcal{B}_r(x^\dagger).
\end{equation*}
\end{enumerate}
\end{assumption}

\begin{theorem}[Convergence rate VI, \protect{\cite[Thm.\,4.3]{Hofmann-Bot}}]\label{thm:HofmannBot}
Let assumption \ref{assumption Bot} hold and $\| y^\delta - y^\dagger \| \leq \delta$. %Let $x_\alpha^\delta$ be a minimizer of the Tikhonov functional in the assumption.
Then for the choice
$
  \alpha \sim \frac{1}{\beta_2} \frac{\psi'}{\phi'}(\delta)
$
it holds
$
  \mathcal{D}_J^\xi(x_\alpha^\delta, x^\dagger) \lesssim \phi(\delta).
$
\end{theorem}

\begin{remark}[Local nonlinearity conditions]\label{rem:localConditions}
For simplicity we assumed the forward operator $F$ to be defined on all of $X$, which is actually not required in the above strand of literature (instead there are conditions on its domain).
Furthermore, it suffices to require the nonlinearity condition only on a sublevelset of $T_{\bar\alpha}^0$ for some fixed but arbitrary $\bar\alpha$ and correspondingly chosen level \cite[Rem.\,3.6]{Hofmann-Kaltenbacher-Poschl-Scherzer}.
Alternatively, one could obviously restrict the Tikhonov functional to a subset $\mathcal B\subset X$ (closed w.r.t.\ $\tau_X$ and containing $x^\dagger$) and then require the nonlinearity condition only on $\mathcal B$.
% Alternatively, one could require the nonlinearity condition only in a small neighbourhood of $x^\dagger$,
% in which case the convergence rates only hold for reconstructions $x_\alpha^\delta$ converging to $x^\dagger$ as $\delta\to0$
% (the Tikhonov reconstructions $x_\alpha^\delta$ may in principle have multiple limit points as the nonlinear inverse problem may have multiple solutions).
\end{remark}

%%---------------------------------------------------------------------------------------------------
\section{Recap of Tikhonov regularization for generalized fidelity terms}\label{sec:generalizedFidelity}

In the previous section the fidelity term of our Tikhonov functionals was based on the norm of Banach space $Y$.
We now recapitulate the extension from \cite{Poschl,Flemming,Werner,Werner-Hohage} to fidelity measures $S(F(x),z)$ that are useful for a bigger class of ill-posed problems.
In particular, they allow the measurement $y^\delta$ to lie in a different space $Z$ than the forward operator $F$ maps into
(as is often relevant for Poisson type data, where the measurement is an empirical measure, while the range of $F$ consists of smooth probability densities).
Hence we define the (generalized) Tikhonov functional as
\begin{align*}
  T_\alpha^\delta (x) \defeq S(F(x),y^\delta) + \alpha J(x).
\end{align*}
Existence, stability, and convergence of the Tikhonov regularization can then be shown under the following assumptions.

\begin{assumption}[Tikhonov functional properties, \cite{Poschl,Flemming,Werner}] \label{setting poschl}
\begin{enumerate}
\item
	Let $X,Y,Z$ be Banach spaces (or affine subspaces) with potentially weaker topologies $\tau_X,\tau_Y,\tau_Z$, let $F:X \to Y$ sequentially continuous w.r.t.\ $\tau_X$ and $\tau_Y$.
  Let $J$ be convex and sequentially lower semi-continuous w.r.t.\ $\tau_X$ with $\tau_X$-sequentially precompact sublevel sets.
	%Consider the forward operator $F: X \to Y$, where $X$ and $Y$ are Banach spaces carrying topologies $\tau_X$ and $\tau_Y$, such that $F$ is sequentially continuous with respect to $\tau_X$ and $\tau_Y$. We also assume that $\mathcal{D}(F)$ is closed in $\tau_X$. Moreover, we assume attainability, meaning, for the exact data $y^\dagger$, there exists a solution $x^\dagger$ of $F(x) = y^\dagger$. The measured data lives in the Banach space $Z$, which also carries a topology $\tau_{Z}$.
	%Here, $J:X \to [0,\infty]$ is assumed to be proper and lower semi-continuous with respect to  $\tau_X$.
\item
  Let $S:Y \times Z \to \R$ satisfy
	\begin{itemize}
		\item $S$ is sequentially lower semi-continuous with respect to  $\tau_Y \otimes\tau_{Z}$,
		\item $S$ is bounded from below on $\mathrm{range}(F)\times Z$,
		\item $z_n \to z$ in $\tau_Z$ implies $S(y,z_n) \to S(y,z)$ for all $y \in Y$ with $S(y,z) < \infty$.% and hence $S(y,z_n)$ is a bounded sequence.
	\end{itemize}
% 	Finally, the condition that replaces the use of Banach--Alouglu: For every $\alpha > 0 $, $z \in Z$ and $M>0$ the sublevel sets
% 	\begin{align*}
% 		\mathcal{M}_{\alpha,z}(M) \defeq \{x \in \mathcal{D}(J) \cap \mathcal{D}(F) : T_\alpha^z(x) \leq M \}
% 	\end{align*}
% 	are sequentially compact with respect to $\tau_X$, that is, every sequence in $ \mathcal{M}_{\alpha,y}(M)$ has a subsequence that converges with respect to $\tau_X$.
% \todo[inline]{The previous coercivity condition is a not so important generalization of requiring that $J$ has precompact sublevel sets, thus we should instead use that simpler condition.}
\end{enumerate}
\end{assumption}

Note that the proofs of the previous convergence rate results \crefrange{Convergence rates Neubauer}{thm:HofmannBot} all exploited that the norm inside the fidelity term of the Tikhonov functional satisfies the triangle inequality
\begin{equation*}
\|F(x_\alpha^\delta)-y^\dagger\|\leq\|F(x_\alpha^\delta)-y^\delta\|+\|y^\dagger-y^\delta\|.
\end{equation*}
The second summand was no larger than $\delta$ by definition, and the first forms part of the Tikhonov functional and thus could be bounded using the minimization property of $x_\alpha^\delta$.
Having Kullback--Leibler-type fidelities $S$ in mind, however, we can no longer assume a triangle inequality of the form $S(F(x),y^\dagger)\leq S(F(x),y^\delta)+S(y^\dagger,y^\delta)$
and then take the last term as quantification of the noise.
In his 2011 dissertation \cite{Werner}, Werner thus described an alternative noise quantification that can be used when the triangle inequality is not available.
He introduced a second functional $\mathcal{T}(\cdot,\cdot)$ to measure the similarity between the reconstruction $F(x)$ and the exact data $y^\dagger=F(x^\dagger)$
(which naturally has to satisfy $\mathcal{T} \geq 0$ and $\mathcal{T}(y,y) = 0$)
and then estimated $\mathcal{T}(F(x_\alpha^\delta),y^\dagger)$ based on the following new noise quantification:
It is assumed that
\begin{equation*}
	\mathrm{Err}(y)(y^\dagger,y^\delta) \leq \delta
  \text{ for all }y\in Y
  \qquad\text{with }
  \mathrm{Err}(y)(y^\dagger,y^\delta) \defeq  \mathcal{T}(y,y^\dagger) - S(y,y^\delta) + S(y^\dagger,y^\delta).
\end{equation*}
With this definition one gets $\mathcal{T}(F(x_\alpha^\delta),y^\dagger)\leq\delta+S(F(x_\alpha^\delta),y^\delta) - S(y^\dagger,y^\delta)$,
of which the last two summands form part of $T_\alpha^\delta(x_\alpha^\delta)-T_\alpha^\delta(x^\dagger)$ and thus can be estimated via the minimization property of $x_\alpha^\delta$.

A possible interpretation of this noise quantification follows from rewriting $\mathrm{Err}(y)(y^\dagger,y^\delta) \leq \delta$ as
\begin{align*}
	\mathcal{T}(y,y^\dagger) - \delta \leq  S(y,y^\delta) - S(y^\dagger,y^\delta).
\end{align*}
The right-hand side essentially is $S$, corrected in such a way that it vanishes when $y$ coincides with $y^\dagger$.
It gives an upper bound for the similarity $\mathcal{T}(y,y^\dagger)$ up to an error $\delta$.
$\mathcal{T}$ should be as strong (\ie large) as possible to give good rates of convergence, but needs to be weaker (\ie smaller) than the right-hand side of the inequality in order to be estimated based on it.
The latter condition might or might not be possible with $\mathcal{T} = S$, which is the reason why the second functional $\mathcal{T}$ became necessary.
Moreover, the fact $\mathcal{T} \geq 0$ enforces $S(y,y^\delta) - S(y^\dagger,y^\delta)$ to stay above $-\delta$, thus the minimal value of $S(\cdot,y^\delta)$ differs from $S(y^\dagger,y^\delta)$ by no more than $\delta$.
In other words, $y^\dagger$ is close to being a minimizer of $S(\cdot,y^\delta)$, which can be understood as similarity between $y^\dagger$ and $y^\delta$ (see \cref{fig:noiseQuantification}).
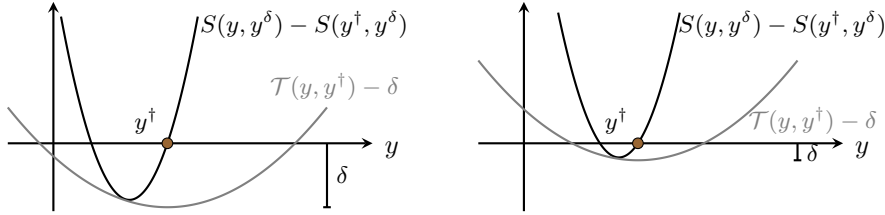
\begin{figure}[h]  % Positioning: here, top, bottom, page
\centering
\begin{tikzpicture}
  \begin{axis}[
    x=1cm,
    y=.75cm,
    axis x line=middle,
    axis y line=middle,
    xtick=\empty,
    ytick=\empty,
    xlabel={$y$},
    ylabel={},
    axis line style=thick,
    ymin=-1.2, ymax=2.5,
    xmin=-0.6, xmax=4.2,
    clip=false,
    every axis x label/.style={at={(current axis.right of origin)},anchor=north west,yshift=4pt,xshift=1pt}
    ]

    % S(y,y^\delta) - S(y^\dagger,y^\delta)
    \addplot[black, thick, samples=200,domain=.1:1.9] {4*(x-1)^2 - 1};
    \node[black] at (axis cs:3.3,2.1) {\small$S(y,y^\delta) - S(y^\dagger,y^\delta)$};

    % \mathcal{T}(y,y^\dagger) - \delta
    \addplot[gray, thick, samples=200,domain=-0.6:3.6] {.4*(x-1.5)^2 - 1.13};
    \node[gray] at (axis cs:3.7,1.0) {\small$\mathcal{T}(y,y^\dagger) - \delta$};

    % x-intercept of first curve: x = 1.5
    \addplot+[mark=*,black,only marks] coordinates {(1.5,0)};
    \node[above left] at (axis cs:1.5,0) {\small$y^\dagger$};

    % Measurement bar from axis
    \draw[thick] (axis cs:3.6,0) -- (axis cs:3.6,-1.13);
    % Add hook-style bars at ends
    \draw[thick] (axis cs:3.55,-1.13) -- (axis cs:3.65,-1.13);
    % Label the bar with delta
    \node[right] at (axis cs:3.6,-0.56) {\small$\delta$};
  \end{axis}
\end{tikzpicture}
\qquad
\begin{tikzpicture}
  \begin{axis}[
    x=1cm,
    y=.75cm,
    axis x line=middle,
    axis y line=middle,
    xtick=\empty,
    ytick=\empty,
    xlabel={$y$},
    ylabel={},
    axis line style=thick,
    ymin=-1.2, ymax=2.5,
    xmin=-0.6, xmax=4.2,
    clip=false,
    every axis x label/.style={at={(current axis.right of origin)},anchor=north west,yshift=4pt,xshift=1pt}
    ]

    % S(y,y^\delta) - S(y^\dagger,y^\delta)
    \addplot[black, thick, samples=200,domain=.45:2.05] {4*(x-1.25)^2 - .25};
    \node[black] at (axis cs:3.4,2.1) {\small$S(y,y^\delta) - S(y^\dagger,y^\delta)$};

    % \mathcal{T}(y,y^\dagger) - \delta
    \addplot[gray, thick, samples=200,domain=-0.6:3.6] {.4*(x-1.5)^2 - .3};
    \node[gray] at (axis cs:3.8,0.4) {\small$\mathcal{T}(y,y^\dagger) - \delta$};

    % x-intercept of first curve: x = 1.5
    \addplot+[mark=*,black,only marks] coordinates {(1.5,0)};
    \node[above left] at (axis cs:1.5,0) {\small$y^\dagger$};

    % Measurement bar from
    \draw[thick] (axis cs:3.6,0) -- (axis cs:3.6,-.3);
    % Add hook-style bars at ends
    \draw[thick] (axis cs:3.55,-.3) -- (axis cs:3.65,-.3);
    % Label the bar with delta
    \node[right] at (axis cs:3.6,-0.17) {\small$\delta$};
  \end{axis}
\end{tikzpicture}
\caption{The sketch illustrates that for smaller $\delta$ the minimal value of $S(\cdot,y^\delta)$ approaches $S(y^\dagger,y^\delta)$. Thus for small $\delta$, $y^\dagger$ \emph{almost} minimizes $S(\cdot,y^\delta)$. Depending on the specific properties of $S$, this implies information on how close $y^\delta$ is to $y^\dagger$.}
\label{fig:noiseQuantification}
\end{figure}

For rates of convergence we need to generalize the previous variational source conditions.
\begin{assumption}[Conditions for convergence rates VII, \cite{Werner-Hohage}]\label{assumption Hohage}
% 	assumption \ref{setting poschl} holds and $J$ is also convex. There exists a constant $\delta \geq 0$ such that
% 	\begin{align*}
% 		\mathrm{Err}(z)(y^\dagger,y^\delta) \leq \delta
% 	\end{align*}
% 	independently of $z$, with $\mathrm{Err} $ as defined in \eqref{eqn:GeneralDefinitionErr}.
\begin{enumerate}
\item
Let $\phi$ be a concave index function (\ie monotonically increasing with $\phi(0)=0$).
\item
Let $x^\dagger$ be a $J$-minimizing solution, and let there exist $\xi \in \partial J(x^\dagger)$ and $\beta_1\in[0,1)$ such that %and some $r>0$
% There exists an index function $\phi$ (\ie $\phi(0) = 0 $ and it is monotonically increasing) such that $\phi$ is concave,  as well as a parameter $\beta > 0$ and $\xi \in \partial J(x^\dagger)$ and some $r>0$, such that for all $ x \in \mathcal{D}(F) \cap \mathcal{B}_r(x^\dagger)$
	\begin{equation} \label{eqn:VariationalSourceCondition}
		%\beta  \mathcal{D}_J^\xi(x, x^\dagger) \leq J(x) - J(x^\dagger) + \phi\left( \mathcal{T}(F(x_\alpha^\delta),F(x^\dagger))\right)
		-\langle\xi,x-x^\dagger\rangle  \leq \beta_1\mathcal{D}_J^\xi(x, x^\dagger) + \phi\left( \mathcal{T}(F(x),y^\dagger)\right)
    \qquad\text{for all }x\in X.
	\end{equation}
\info{is equivalent to original formulation under $\beta=1-\beta_1$}
\end{enumerate}
\end{assumption}

% In order to formulate the general convergence result, we need the Fenchel conjugate of a function $\gamma: \mathbb{R} \to (-\infty,\infty]$, namely
% \begin{align*}
% 	\gamma^*(s) = \sup_{\tau \in \mathbb{R}} (s \tau - \gamma(\tau)).
% \end{align*}

\begin{theorem}[Convergence rate VII, \protect{\cite[Thm.\,3.3]{Werner-Hohage}}] \label{convergence rates Hohage}
	Let assumptions \ref{setting poschl} and \ref{assumption Hohage} hold and $\mathrm{Err}(y)(y^\dagger,y^\delta) \leq \delta$ for all $y\in Y$. %let $x_\alpha^\delta$ be a minimizer of the Tikhonov functional in the assumption.
  Then for the choice $-1/\alpha\in\partial (-\phi) (\delta)$ it holds
%   Then for all $\alpha>0$ we have
% 	\begin{align*}
% 		\beta \mathcal{D}_J^\xi(x_\alpha^\delta, x^\dagger) \leq \frac{ \delta}{\alpha} + (-\phi)^*(-\frac{1}{\alpha}).
% 	\end{align*}
% 	Moreover, the optimal choice $\hat\alpha$, \ie the minimizer of the right-hand side, is given if and only if
% 	\begin{align*}
% 		-\frac{1}{\hat\alpha} \in \partial (-\phi) (\delta),
% 	\end{align*}
% 	which gives the optimal rate
	$
		\mathcal{D}_J^\xi(x_{\alpha}^\delta, x^\dagger) \leq \phi(\delta)/(1-\beta_1).
	$
\end{theorem}

%%---------------------------------------------------------------------------------------------------
\section{Forward operator and fidelity term for drift estimation in SDE}\label{sec:forwardOp}

Recall that at fixed time points $0=t_0< t_1< \ldots< t_M = T$ we observe the position of $n$ distinguishable particles whose random motion in the smooth domain $\Omega\subset\R^d$ is governed by SDE \eqref{eqn:SDE}.
From this observation we aim to estimate the drift $\mu:\Omega\to\R^d$ (while $\sigma>0$ is assumed known).
We denote by $q_j=(q_j^0,\ldots,q_j^M)\in\Omega^{M+1}$ the measured position of the $j$th particle at times $t_0,\ldots,t_M$.
The $q_j^i$ are realizations of random variables $Q_j^i$, where $Q_j = (Q_j^0,\ldots,Q_j^M)$ are independent and identically distributed for different $j$.
For a MAP estimate of $\mu$ we will need to evaluate the log-likelihood of the measurements, for which in turn we need to find the density (\eg w.r.t.\ the Lebesgue measure) of the joint law of all $Q_j^i$.
For fixed $j$ let the probability measure
\begin{equation*}
	\nu(\d x^0,\ldots,\d x^M)
\end{equation*}
on $\Omega^{M+1}$ denote the joint law of $(Q_j^0,\ldots,Q_j^M)$.
By the Markov property of the process, the law of $Q_j^{i+1}$ only depends on the realization of $Q_j^i$ but is independent of the realizations of $Q_j^0,\ldots,Q_j^{i-1}$.
Thus, by the disintegration theorem backwards in time we obtain
\begin{equation*}
	\nu (\d x^0, \ldots,\d x^M) = \nu_{x^{M-1}} (\d x^M)\nu_{x^{M-2}} (\d x^{M-1}) \cdot\cdot\cdot \nu_{x^0} (\d x^1) \hat\nu_0(\d x^0),
\end{equation*}
where $\hat\nu_0$ denotes the law of $Q_j^0$ (which we assume to have density $g^0>0$ w.r.t.\ the Lebesgue measure)
and the measure $\nu_{x^{i-1}} (\d x^i)$ describes the law of $Q_j^i$ given $Q_j^{i-1} = x^{i-1}$.
Its Lebesgue density is given by the solution at time $t_i$ of the Fokker--Planck equation associated with the SDE,
\begin{equation}\label{eqn:FokkerPlanck}
\partial_tp=-\div(\mu p)+\Delta(\tfrac{\sigma^2}2p)
\qquad\text{in }[0,T]\times\Omega,
\end{equation}
with (total mass preserving) homogeneous Neumann boundary conditions
\begin{equation}\label{eqn:bc}
\tfrac{\sigma^2}2\nabla p\cdot\nu=(\tfrac{\sigma^2}2\nabla p-\mu p)\cdot\nu=0
\qquad\text{on }\partial\Omega\text{ for }\nu\text{ the outward normal}
\end{equation}
(recall $\mu\cdot\nu=0$ on $\partial\Omega$)
and for initial condition $p(t_{i-1})=\delta_{x^{i-1}}$ a Dirac measure in $x^{i-1}$.
This is exactly the Green's function $g_\mu(t_i,x^i;t_{i-1},x^{i-1})$ of the Fokker--Planck equation, thus we can write
\begin{equation*}
	\nu (\d x^0, \ldots,\d x^M) =g^0(x^0) \prod_{i=1}^M g_\mu(t_i,x^i;t_{i-1},x^{i-1}) \,\d x^0\cdot\cdot\cdot \d x^M.
\end{equation*}
The likelihood $S_L(q)$ of $q=(q_1,\ldots,q_n)$ (w.r.t.\ the Lebesgue measure on $(\Omega^{M+1})^n$) and the log-likelihood $S_l(q)$ thus read
% Having the density of $\nu$ at hand, we can define the likelihood of the density $g^0(x^0) \prod_{i=1}^M g_\mu(t_i,x^i;t_{i-1},x^{i-1})$ given the measurements $q_j = (q^0,\ldots,q^M)_j$, namely
\begin{equation*}
	S_L(q) =  \prod_{j=1}^n g^0(q^0_j) \prod_{i=1}^M g_\mu(t_i,q^i_j;t_{i-1},q^{i-1}_j),\qquad
% \end{align*}
% and the log-likelihood
% \begin{align*}
	S_l(q) =  \sum_{j=1}^n \left[\ln g^0(q_j^0) + \sum_{i=1}^M \ln g_\mu(t_i,q^i_j;t_{i-1},q^{i-1}_j)\right].
\end{equation*}
We now define the forward operator as
\begin{equation*}
	F : L^\infty(\Omega;\mathbb{R}^d) \longrightarrow \left(L^2(\Omega \times \Omega) \right)^{M},
  \qquad
	\mu \mapsto \begin{bmatrix}
		g_\mu(t_1,\cdot;t_0,\cdot) \\
		g_\mu(t_2,\cdot;t_1,\cdot) \\
		\vdots \\
		g_\mu(t_M,\cdot;t_{M-1},\cdot)
	\end{bmatrix},
\end{equation*}
and the measurement as the empirical probability measure of the particle distribution on $\Omega^{M+1}$,
\begin{align*} 
	G^n \defeq \frac{1}{n} \sum_{j=1}^n \delta_{q_j}.
\end{align*}
The fidelity term $S(F(\mu),G^n)$ in a MAP functional is given by the negative log-likelihood, thus
\begin{equation*}
S(F(\mu),G^n)
=-S_l(q)/n
= - \int_{\Omega^{M+1}} \left[\ln g^0(x^0) + \sum_{i=1}^M \ln g_\mu(t_i,x^i;t_{i-1},x^{i-1})\right]\,\d G^n(x^0,\ldots,x^M)
\end{equation*}
(note that for notational simplicity we rescaled the log-likelihood by a constant factor, which just corresponds to considering the likelihood w.r.t.\ a rescaled base measure).
As function space setting for $S$ we can for instance pick the following:
Let $\mathcal{MZ}(\Omega)$ denote the convex set of Markov kernels from $\Omega$ to $\Omega$ which are absolutely continuous w.r.t.\ the Lebesgue measure.

Note that $F_i(\mu)$ lies in that set, since the map $F_i(\mu)(\cdot,x) = g_\mu(t_i,\cdot;t_{i-1},x)$ is a probability densitiy on $\Omega$ for every $x \in \Omega$.
Note also that if drift $\mu$ in \eqref{eqn:FokkerPlanck} is divergence-free, then $F_i(\mu)$ is even a doubly stochastic Markov kernel.
Indeed, this can be seen from $g_\mu(t_i,x;t_{i-1},\cdot)$ being the Green's function for the adjoint equation to \eqref{eqn:FokkerPlanck}, which for $\div\mu=0$ is of the same type (except for being backward parabolic).
Equivalently one can check that the solution $p(t,x)$ to \eqref{eqn:FokkerPlanck} for constant initial condition $p(t_{i-1},\cdot)\equiv1$
satisfies $p(t,\cdot)\equiv1$ if and only if $\div\mu=0$, thus also $1=p(t_i,x)=\int g_\mu(t_i,x;t_{i-1},x^{i-1}) \,\d x^{i-1}$, implying the double stochasticity.
Let further $\mathcal{P}(\Omega^{M+1})$ be the set of probability measures on $\Omega^{M+1}$ and notice that $\mathcal{MZ}(\Omega)^M$ can be understood as a subset of $\mathcal{P}(\Omega^{M+1})$, since for $(z_1,\ldots,z_M)^T \in \mathcal{MZ}(\Omega)^M$ the product $z^\pi \defeq g^0\prod_{i=1}^M  z_i$ is a probability density on $\Omega^{M+1}$ w.r.t.\ the Lebesgue measure. Now we set
\begin{align*}
	S: \mathcal{MZ}(\Omega)^M \times &\mathcal{P}(\Omega^{M+1}) \to \mathbb{R},\qquad
	\left(\begin{pmatrix}
		z_1 \\
		\vdots \\
		z_M
	\end{pmatrix} , G^n\right) \mapsto \begin{cases}
		-\int_{\Omega^{M+1}} \ln  z^\pi  \,\d G^n & z^\pi >  0   \text{ $G^n$-a.e.}\\
		\infty & \text{else}.
	\end{cases} %\\
% 	&= \begin{cases}
% 		-\int  \left[\ln g^0(q_j^0) + \sum_{i=1}^M \ln g_\mu(t_i,q^i_j;t_{i-1},q^{i-1}_j)\right] \,\d G^n & z >  0   \textit{ a.e.}\\
% 		\infty & \textit{else},
% 	\end{cases}
\end{align*}
Let us further recall the Kullback--Leibler divergence $\mathrm{KL}(z,u)=\int u\ln\frac uz+z-u\,\d x$ and consider its version
\begin{equation*}
	\mathrm{KL}^\pi: \mathcal{MZ}(\Omega)^M \times \mathcal{MZ}(\Omega)^M \to \mathbb{R},\quad
	\left(\begin{pmatrix}
		z_1 \\
		\vdots \\
		z_M
	\end{pmatrix} , \begin{pmatrix}
		u_1 \\
		\vdots \\
		u_M
	\end{pmatrix}\right) \mapsto  \mathrm{KL}(z^\pi,u^\pi)
	=\int_{\Omega^{M+1}} \left(u^\pi \ln\frac{u^\pi}{z^\pi} + z^\pi - u^\pi\right)\, \d x.
\end{equation*}
$\mathrm{KL}^\pi$ will play the role of $\mathcal T$ from the previous section,
so the error functional to measure the difference between the exact data $g^\dagger\defeq F(\mu^\dagger) \in \mathcal{MZ}(\Omega)^M$ and the measurement $G^n \in \mathcal{P}(\Omega^{M+1})$ is taken as
\begin{align}\label{eqn:Err}
	\mathrm{Err}(y)(g^\dagger,G^n) = \mathrm{KL}^\pi(y,g^\dagger) - S(y,G^n) + S(g^\dagger,G^n) = \int \ln \frac{y^\pi}{(g^\dagger)^\pi} \ (\d G^n - (g^\dagger)^\pi \d x).
\end{align}    
To use the general convergence rate result from last section, we would need to bound $\mathrm{Err}(y)(g^\dagger,G^n)$ independently of $y$.
To this end we will in the next section need to slightly alter the fidelity term.
Also, since $G^n$ is a random variable, we can only hope for a probabilistic bound so that in turn the resulting convergence statement (which holds under the variational source condition \eqref{eqn:VariationalSourceCondition}) will be probabilistic.

%%---------------------------------------------------------------------------------------------------
\section{Probabilistic convergence rates under stochastic bounds on the noise}\label{sec:probabilisticRate}
Here we aim to apply \cref{convergence rates Hohage}
(or, appealing to remark \ref{rem:localConditions}, its version when restricting the reconstruction to some $\mathcal B$)
to our specific setting in which the noise can only be quantified in a probabilistic sense.
To this end we can follow \cite{Dunker-Hohage}, who considered the stationary version of our problem.
We first specialize assumption \ref{setting poschl} towards our setting as follows:
\begin{enumerate}
	\item $(X,\| \cdot \|_X)$ is a Banach space and $\mathcal{B} \subset X$ is a convex subset and $\tau_X$ is the weak topology on $X$.
  (In our specific parameter identification problem, $\mathcal B$ is the set of allowed drifts on $\Omega$ and will later be taken as a ball in some Hilbert--Sobolev space.)
  \info{for us later $(H^r,\| \|_{H^r})$ with $r$, such that a drift here is also in $L^\infty$. $r>d/2 +1$ implies $C^{1,\alpha}$ Alt Thm.\,8.13. So we get the estimates and If $\Omega$ bounded we get also $L^\infty$. $r$ also has to be large enough, such that weak convergence in $H^r$ implies strong convergence in $H^{\tilde{r}}$, where $\tilde{r}$ is relevant in the continuity of $F$.}
	
	\item $Y = (H^s(\Omega \times \Omega)^M,\| \cdot \|_{H^s(\Omega \times \Omega)^M})$ for a bounded smooth domain $\Omega \subset \mathbb{R}^d$ (in fact, within this section Lipschitz regularity would be sufficient), where $s > \frac{(M+1)d}{2}$. Moreover, $\tau_Y$ is the strong topology in $Y$. \info{ $s > \frac{(M+1)d}{2}$ implies $C^{0,\alpha}$ with Alt}
	
	\item $F: (X,\| \cdot \|_X) \supset \mathcal{B} \longrightarrow Y$ is continuous with respect to $\tau_X$ and $\tau_Y$.
	\info{for us the continuity w.r.t. $L^\infty$ implies continuity w.r.t. $\tau_X$ since this is a stronger topology, again if $r$ is large enough. We actually proved continuity from $H^{\tilde{r}}$ to $H^s$ for an $\tilde{r} \geq s, d/2 +1$, for this reason we choose $r$, such that $\tau_X$ is stronger than the strong topology of $H^{\tilde{r}}$.}
	
	\item We assume $\sup_{y \in \mathcal{B}} \| F(y) \|_{H^s(\Omega \times \Omega)^M} < Q$ for some $Q > 0$.
	\item The regularization functional $J:\mathcal B\to \mathbb{R}$ is convex and sequentially lower semi-continuous with precompact sublevel sets in the topology $\tau_X$. \info{for us $J = \| \|_{H^r}^2$, so it holds with $\tau_X$ being weak topology in $H^r$}

  \item
  We take $Z=\mathcal P(\Omega^{M+1})$ the probability measures on $\Omega^{M+1}$ with $\tau_Z$ the weak-* topology.

  \item
  The fidelity functional $S_\tau$ (which will replace $S$ and will be introduced momentarily) satisfies assumption \ref{setting poschl} (2).

	\item In addition we will make use of the fact that $F(x)\in\mathcal{MZ}(\Omega)^M$ for all $x \in \mathcal B$, and we assume $g^\dagger\in F(\mathcal B)$ and $g^0 \in H^s(\Omega)\cap\mathcal P(\Omega)$ (\ie a probability density of $H^s$-regularity).% such that the probability density $g^0\prod_{i = 1}^{M}F_i(x)$ on $\Omega^{M+1}$ is in $H^s(\Omega^{M+1})$.
\end{enumerate}

To bound $\mathrm{Err}(y)(g^\dagger,G^n)$, we first note that $G^n$ is the empirical measure of $n$ realizations of the measure $(g^\dagger)^\pi \d x$. Hence to estimate the right-hand side of \eqref{eqn:Err} we consider the following result.
\begin{theorem}[Concentration inequality, \protect{\cite[Cor.\,5]{Dunker-Hohage}}] \label{corollary Massart Hohage}
	Let $B_{\mathfrak r}^s$ be the ball of radius ${\mathfrak r}$ in $H^s(\Omega^{M+1})$ with $s>(M+1)d/2$. There exists a constant $C \geq 1$ depending only on $\Omega$ and $s$ such that for $\rho \geq{\mathfrak r}C$ and for all $n \in \N$ it holds
	\begin{align*}
		\mathbb{P} \left[ \sup_{y \in B_{\mathfrak r}^s} \left| \int y (d G^n - (g^\dagger)^\pi \d x)\right|  \geq \frac{\rho}{\sqrt{n}}  \right] \leq \exp{\left(-\frac{\rho}{\mathfrak r C}\right)}.
	\end{align*}
\end{theorem}
\info{One can probably reduce the condition on $s$ to $s>d/2$ on the expense of exploiting that $y$ is of the form $\ln(\tau+\text{product of $H^s$-functions})/(\tau+\text{product of $H^s$-functions})$.}
However we cannot use this result directly to bound \eqref{eqn:Err}, because $\ln \frac{z^\pi}{(g^\dagger)^\pi}$ is not bounded without assuming the reconstuctions $z^\pi$ to be bounded away from zero. For this reason we consider a modification of our setting. From now on, as in \cite{Dunker-Hohage}, we will work with the following functionals modified by a shift parameter $\tau > 0$,
\begin{align*}
	S_\tau\!:\mathcal{MZ}(\Omega)^M \!\times\! \mathcal{P}(\Omega^{M+1}) &\to \mathbb{R},&&&
	S_\tau(y,G^n) &\defeq -\int \ln  (y^\pi +\tau) \  (G^n +\tau \d x),\\
	\mathrm{KL}^\pi_\tau\!:\mathcal{MZ}(\Omega)^M \times \mathcal{MZ}(\Omega)^M &\to \mathbb{R},  &&&
	\mathrm{KL}^\pi_\tau(y,u) &\defeq\mathrm{KL}(y^\pi\!\!+\!\tau,u^\pi\!\!+\!\tau)= \!\int\! (u^\pi\!\!+\!\tau) \ln\frac{u^\pi\!\!+\!\tau}{y^\pi\!\!+\!\tau} \!+\! y^\pi \!-\! u^\pi\,\d x.
\end{align*}
\begin{remark}[Properties of $S_\tau$]
Note that $S_\tau$ indeed satisfies the properties of assumption \ref{setting poschl} (2) if $g^0\in H^s(\Omega)$ and $g^\dagger\in H^s(\Omega\times\Omega)^M$.
Indeed, for $y_n\to y$ in $H^s(\Omega\times\Omega)^M$ and thus by Sobolev--H\"older embedding also in $C^0(\Omega\times\Omega)^M$ we obtain $y_n^\pi\to y^\pi$ in $C^0(\Omega^{M+1})$
and thus $\ln\frac{y_n^\pi+\tau}{(g^\dagger)^\pi+\tau}\to\ln\frac{y^\pi+\tau}{(g^\dagger)^\pi+\tau}$ in $C^0(\Omega^{M+1})$.
Together with $G^n\stackrel*\rightharpoonup g$ one obtains $S_\tau(y_n,G^n)\to S_\tau(y,g)$ so that $S_\tau$ is continuous w.r.t.\ $\tau_Y\otimes\tau_Z$.
Moreover, the range of $F$ over $\mathcal B$ contains only uniformly bounded nonnegative continuous functions so that $S_\tau$ is bounded below.
\end{remark}
We correspondingly modify the Tikhonov functional to
\begin{align} \label{eqn:TikhonovFunctionalShifted}
	T_{\alpha}^n(x) \defeq S_\tau(F(x), G^n) + \alpha J(x)\qquad\text{for }x\in\mathcal B.
\end{align}
The advantage of introducing the shift parameter $\tau$ becomes clear when calculating $\mathrm{Err}$ for the shifted fidelity functional,
\begin{align*}
	\mathrm{Err}_\tau(y)(g^\dagger,G^n) &\defeq \mathrm{KL}_\tau(y,g^\dagger) - S_\tau(y,G^n) + S_\tau(g^\dagger,G^n) = \int \ln \frac{y^\pi + \tau}{(g^\dagger)^\pi + \tau} \ (G^n - (g^\dagger)^\pi \d x).
\end{align*}
This can indeed be bounded probabilistically, since all probability densities $y$ fulfil $y^\pi +\tau > \tau > 0$ and thus $\ln \frac{y^\pi + \tau}{(g^\dagger)^\pi + \tau}\in H^s$ if $y^\pi$ and $g^\dagger$ are. %It follows as a consequence of corollary \ref{corollary Massart Hohage} in fact
\begin{corollary}[Probabilistic shifted error bound] \label{corollary probabilistic bound for Err_sigma}
	Let $g^\dagger\in B_Q^s$, the ball of radius $Q$ in $H^s(\Omega\times\Omega)^M$ with $s>(M+1)d/2$. There exists $C \geq 1$ depending only on $\Omega$, $s$, and $g^0\in H^s(\Omega)$ such that for $\rho \geq Q^MC$ and for all $n \in \N$ it holds
	\begin{align*}
		\mathbb{P} \left[ \sup_{y \in B_Q^s,y\geq0} \left| \mathrm{Err}_\tau(y)(g^\dagger,G^n) \right|  \geq \frac{\rho}{\sqrt{n}}  \right] \leq \exp{\left(-\frac{\rho}{Q^MC}\right)}.
	\end{align*}    
\end{corollary}
\begin{proof}
This is a direct consequence from the fact that boundedness and nonnegativity of $y$ in $H^s(\Omega\times\Omega)^M$
implies boundedness and nonnegativity of $y^\pi$ in $H^s(\Omega^{M+1})$ (since multiplication of functions in $H^s(\Omega\times\Omega)$ is bounded for $s>\frac d2+1$)
and thus boundedness of $\ln \frac{y^\pi + \tau}{(g^\dagger)^\pi + \tau}=\ln(y^\pi)-\ln((g^\dagger)^\pi+\tau)$ (since $z\mapsto\ln(z+\tau)$ is smooth for $z\geq0$).
Hence \cref{corollary Massart Hohage} can be applied for an appropriate choice of ${\mathfrak r}$.
\end{proof}

Assuming the variational source condition, this allows to obtain the following result by Dunker and Hohage about the convergence rate in expectation.
\begin{theorem} [Convergence rate in expectation, \protect{\cite[Thm.\,6]{Dunker-Hohage}}]\label{convergence rates Dunker Expectation}
  Let assumption \ref{assumption Hohage} hold for $\mathcal T=\mathrm{KL}^\pi_\tau$ and $y^\dagger=g^\dagger$, and let
	\begin{equation*}
		a \defeq \left.\frac{Q^MC}{1-\beta_1} \sum_{k=1}^\infty k\exp (-(k-1))\right/\sum_{k=1}^\infty \exp (-(k-1))>1
    \qquad\text{for $C$ from corollary \ref{corollary probabilistic bound for Err_sigma}.}
	\end{equation*}
  Then for the choice $-1/\alpha\in\partial (-\phi) (\frac{a}{\sqrt{n}})$, the minimizer $x_\alpha^n$ of \eqref{eqn:TikhonovFunctionalShifted} satisfies
% 	Let $x_\alpha^n$ be a minimizer of the Tikhonov functional in \eqref{eqn:TikhonovFunctionalShifted} with $n$ denoting the number of observations. Let $x^\dagger$ be a $J$-minimizing solution of $F(x) = g^\dagger$ and let it fulfil a variational source condition of the form
% 	\begin{align} \label{eqn:VariationalSourceConditionKL}
% 		\beta  \mathcal{D}_J^\xi(x, x^\dagger) \leq J(x) - J(x^\dagger) + \phi\left( \mathrm{KL}^\pi_\tau(F(x) , g^\dagger)\right) \textit{ for all } x \in \mathcal{D}(F)
% 	\end{align}
% 	with $\phi, \xi$ and $\beta$ as in assumption \ref{assumption Hohage}. Let $\alpha$ be chosen such that
% 	\begin{align*}
% 		-\frac{1}{\alpha} \in \partial (-\phi) (\frac{a}{b\sqrt{n}}),
% 	\end{align*}
% 	where
% 	\begin{align*}
% 		a \defeq \frac{KC}{\beta} \left(\sum_{k=1}^\infty k\exp (-(k-1))\right) \  \textit{ and } \   b \defeq \left(\sum_{k=1}^\infty \exp (-(k-1)) \right)
% 	\end{align*}
% 	Then
	$
		\mathbb{E} \left[ \mathcal{D}_J^\xi(x_\alpha^n, x^\dagger) \right] \lesssim \phi\left( \frac{a}{\sqrt{n}} \right)\lesssim\phi\left( \frac1{\sqrt{n}} \right).
	$
\end{theorem}
\info{Last inequality follows from $\phi(x)/x$ being monotonically decreasing.}
\info{To reduce the conditions on $s$ one could consider the following setting, however, then lemma \ref{thm:KLdivergenceEstimate} fails:
Introduce $u_\tau^\pi=g^0\prod_i(u_i+\tau)$ and make $\mathrm{KL}_\tau^\pi(y,u)=\mathrm{KL}(y_\tau^\pi,u_\tau^\pi)$ (and similarly adapt the fidelity term), then use that
$\mathbb{P} \left[ \sup_{y \in B_Q^s} \left| \mathrm{Err}_\tau(y)(g^\dagger,G^n) \right|  \geq \frac{\rho}{\sqrt{n}}  \right]$
is less than probability that any of the $i$ components is greater than $\rho/\sqrt n/M$
}

%%---------------------------------------------------------------------------------------------------
\section{Reducing the variational source condition to a tangential cone condition}\label{sec:reduction}
The task now is to understand better when the variational source condition \ref{assumption Hohage} for $\mathcal T=\mathrm{KL}_\tau^\pi$ holds.
Following \cite{Dunker-Hohage}, it turns out that, choosing the regularization term $J$ to be a sufficiently high Hilbert--Sobolev norm squared, the variational source condition reduces to the following simpler nonlinearity condition (which was introduced in \cite{Hanke_Neubauer_Scherzer} in the context of Landweber iteration),
called \emph{tangential cone condition} in \cite{Dunker-Hohage},
\begin{equation} \label{eqn:TangentialConeCondition}
	\left\| F(\mu^\dagger\!\!+\!h) \!-\! F(\mu^\dagger) \!-\! F'(\mu^\dagger)[h] \right\|_{L^2(\Omega \times \Omega)^M} \lesssim \| h \|_{L^\infty(\Omega;\R^d)} \| F(\mu^\dagger\!\!+\!h) \!-\! F(\mu^\dagger) \|_{L^2(\Omega \times \Omega)^M}
  \ \ \text{for all }h\in B_R^r,\!
\end{equation}
where again $B_R^r$ is a ball of some radius $R$ in $H^r(\Omega;\R^d)$ and $r$ is chosen large enough to make $F$ continuous from $H^r(\Omega;\R^d)$ to $Y$
(here, $F'(\mu^\dagger)$ stands for the Fr\'echet derivative of $F$ in $\mu^\dagger$ as a map from $L^\infty(\Omega)$ to $L^2(\Omega\times\Omega)^M$).
Indeed,
making the simplifying, yet natural assumption that
\begin{equation}\label{eqn:divergenceFreeAssumption}
\text{all time intervals $[t_{i-1},t_i]$ have equal length \qquad or \qquad all admissible drifts $\mu$ are divergence-free,}
\end{equation}

one can prove the following implication.
\begin{theorem}[Tangential cone implies variational source condition]\label{thm:tangentialConeImpliesVariationalSource}
	Let $F$ fulfil conditions \eqref{eqn:TangentialConeCondition}, \eqref{eqn:divergenceFreeAssumption} and $J(\mu) = \| \mu \|_{H^{r}}^2$ with $r>\frac{d}{2} + 1$, then every admissible $\mu^\dagger \in H^{r}(\Omega,\mathbb{R}^d)$ satisfies a variational source condition of the form
	\begin{align*}
		%\beta  \mathcal{D}_J^\xi(\mu, \mu^\dagger) \leq J(\mu) - J(\mu^\dagger) + \phi \left(\mathrm{KL}_\tau^\pi(F(\mu),u^\dagger) \right)
		-\langle\xi,\mu-\mu^\dagger\rangle  \leq \beta_1\mathcal{D}_J^\xi(\mu, \mu^\dagger) + \phi \left(\mathrm{KL}_\tau^\pi(F(\mu),g^\dagger) \right)
    \qquad\text{for all admissible }\mu-\mu^\dagger\in B^r_R
	\end{align*}
	so that \cref{convergence rates Dunker Expectation} can be applied.
\end{theorem}
Note that here we only obtain the variational source condition in a ball around $\mu^\dagger$,
so to apply the convergence rate results from the previous section we need to choose $\mathcal B$ as that ball (see also remark \ref{rem:localConditions}).
The proof of \cref{thm:tangentialConeImpliesVariationalSource} is identical to that of \cite[Prop.\,11]{Dunker-Hohage} up to minor modifications due to our slightly different setting:
For us, the functional $\mathrm{KL}_\tau^\pi$ plays the role of $\mathrm{KL}_\tau$ in the original proof, in which it was used to bound the $L^2$-norm as follows.
\begin{lemma}[Estimate for Kullback--Leibler divergence, \protect{\cite[Thm.\,2.3]{Borwein}}] \label{L^2--KL relationship}
	For all nonnegative functions $x, y \in L^\infty(D)$ with $x-y \in L^2(D)$ and $y>0$ a.e.\ it holds
	\begin{align*}
		\| x-y \|_{L^2(D)}^2 \leq 2 \max(\|x\|_{L^\infty(D)}, \|y\|_{L^\infty(D)}) \mathrm{KL}(x,y).
	\end{align*}
\end{lemma}
Therefore we have to replace this estimate in our setting by the following one (which is a consequence of lemma \ref{L^2--KL relationship}).
\begin{lemma}[Estimate for Kullback--Leibler divergence]\label{thm:KLdivergenceEstimate}
	Let $z=F(\mu_z), u=F(\mu_u)  \in\mathcal{MZ}(\Omega)^M$ be in the range of operator $F$ and assume \eqref{eqn:divergenceFreeAssumption}, then the following holds:
	\begin{equation*}
		\| z - u  \|_{{L^2(\Omega \times \Omega)}^M}^2
		\leq\tfrac{2M|\Omega|^{M-1}}{\inf_{x\in\Omega}g^0(x)^2}\max(\|u^\pi+\tau\|_{L^\infty(\Omega^{M+1})}, \|z^\pi+\tau\|_{L^\infty(\Omega^{M+1})})\mathrm{KL}_\tau^\pi(z,u).
	\end{equation*}
\end{lemma}
\begin{proof}%[Proof of lemma\,\ref{thm:KLdivergenceEstimate}]
Let us first assume that $\div \mu_z = \div \mu_u = 0$ (and thus each component of $u$ and $z$ is doubly stochastic, \ie $\int_\Omega z_i(x,\hat x)\,\d x=1=\int_\Omega z_i(x,\hat x)\,\d\hat x$ and analogously for $u$), then it holds
\begin{align}\label{eqn:MarginalDensity}
z_i(x^i,x^{i-1})
&=\int_{\Omega}z_{i-1}(x^{i-1},x^{i-2})\cdots \int_\Omega z_{1}(x^{1},x^0)\,\d x^0\cdots\d x^{i-2} \nonumber \\
&\quad\cdot z_i(x^i,x^{i-1})\cdot
\int_{\Omega}z_{i+1}(x^{i+1},x^i)\cdots \int_\Omega z_{M}(x^{M},x^{M-1})\,\d x^M\cdots\d x^{i+1}  \\
&=\int_{\Omega^{M-1}}z^\pi(x^0,\ldots,x^M)/g^0(x^0)\,\d x^0\cdots\d x^{i-2}\,\d x^{i+1}\cdots\d x^M \nonumber
\end{align}
for all $i=1,\ldots,M$. Thus, with Jensen's inequality we obtain
\begin{align*}
	\| z_i - u_i \|_{L^2(\Omega \times \Omega)}^2
  &= \int_{\Omega^2} \left(\int_{\Omega^{M-1}}\!\!\!\frac{z^\pi(x^0,\ldots,x^M)-u^\pi(x^0,\ldots,x^M)}{g^0(x^0)}\,\d x^0\cdots\d x^{i-2}\,\d x^{i+1}\cdots\d x^M\right)^2 \d x^{i-1} \,\d x^i  \\
  &\leq|\Omega|^{M-1} \int_{\Omega^{M+1}}\!\!\!\frac{|z^\pi(x^0,\ldots,x^M)-u^\pi(x^0,\ldots,x^M)|^2}{g^0(x^0)^2}\,\d x^0\cdots\d x^M \\
	&\leq|\Omega|^{M-1}\| z^\pi-u^\pi \|_{L^2(\Omega^{M+1})}^2/\inf_{x\in\Omega}g^0(x)^2.
\end{align*}
The desired estimate now follows from lemma\,\ref{L^2--KL relationship} and the definition of $\mathrm{KL}_\tau^\pi$. Finally, if instead of divergence-free drifts we assume equally long time intervals, all components of $z$ coincide with the first component $z_1$ (analogously for $u$), for which \eqref{eqn:MarginalDensity} still holds. Thus the result follows for all $i = 1,\ldots,M$.
\end{proof}
This estimate is applied to $g^\dagger$ and $F(\mu)$ so that we have to bound $\|(g^\dagger)^\pi+\tau\|_{L^\infty(\Omega^{M+1})}$ and $\|F(\mu)^\pi+\tau\|_{L^\infty(\Omega^{M+1})}$ by a constant;
as in the previous section this follows from the boundedness of $\mathrm{range}(F)$ on $\mathcal B$ in $H^s(\Omega\times\Omega)^M$ and Sobolev embedding.
The remainder of the proof is identical.
Note that it exploits the fact by Flemming and Hofmann \cite[Thm.\,3.1]{Flemming_Hofmann} that general source conditions imply variational source conditions provided the tangential cone condition (or even a weaker variant) is fulfilled,
together with the following fact by Hofmann and Math\'e stating that in Hilbert spaces, general source conditions can always be fulfilled.
\begin{theorem}[General classical source conditions, \protect{\cite{mathe}}]
	Let $K: X \to X $ be a compact, self adjoint, injective and nonnegative linear operator. Then for every $x \in X$  there is an index function $\Theta$ such that $x = \Theta(K)w$ for some $w \in X$.
\end{theorem}
In the proof of \cite[Prop.\,11]{Dunker-Hohage} this statement is applied for the operator $K=F'(\mu^\dagger)^\#F'(\mu^\dagger)$,
where their forward operator $F$ maps a drift $\mu$ onto the resulting equilibrium distribution of particles (the solution of \eqref{eqn:FokkerPlanck} after infinite time);
it seems to us that this actually requires an additional argument since the corresponding $F'(\mu^\dagger)$ is not injective (its kernel should contain all vector fields $h$ with $\div(F(\mu^\dagger)h)=0$).
In our setting, on the other hand, the kernel of $F'(\mu^\dagger)$ is empty (it would contain all infinitesimal drift perturbations $h$ that leave the final time solution of \eqref{eqn:FokkerPlanck} invariant,
\emph{irrespective of the initial condition}).

%\notinclude

\section{A weaker tangential cone condition}\label{sec:tangentialCone}

In this section we will prove that in our parameter identification setting a weaker version of \eqref{eqn:TangentialConeCondition} holds for all sufficiently regular $\mu^\dagger$, \ie for all $\mu^\dagger$ in the space
\begin{equation*}
H^r_{\mathrm N}=\{\mu\in H^r(\Omega;\R^d)\,|\,\mu\cdot\nu=0\text{ on }\partial\Omega\}
\end{equation*}
of drifts with vanishing normal component on the boundary, where $r>1+\frac d2$.
In more detail, we extend forward operator $F$ to a time-dependent operator $P$ with $F(\mu)_i=P(\mu)(t_i-t_{i-1})$ and then prove
\begin{equation} \label{eqn:weakTangentialConeCondition}
	\left\| P(\mu^\dagger\!\!+\!h) \!-\! P(\mu^\dagger) \!-\! P'(\mu^\dagger)[h] \right\|_{L_\alpha^\infty L^2(\Omega \times \Omega)} \lesssim \| h \|_{L^\infty(\Omega;\R^d)} \| P(\mu^\dagger\!\!+\!h) \!-\! P(\mu^\dagger) \|_{L_\alpha^\infty L^2(\Omega \times \Omega)}
  \ \ \text{for all }h\in H^r_{\mathrm N}\!
\end{equation}
with some $\alpha>0$ (and the involved constant depending only on $\|\mu^\dagger\|_{H^r}$),
where for a function space $X$ we abbreviate
\begin{equation*}
L_\alpha^\infty X=L_\alpha^\infty((0,T);X)=\{f:(0,T)\to X\,|\,\|f\|_{L_\alpha^\infty X}<\infty\}
\qquad\text{with}\quad
\|f\|_{L_\alpha^\infty X}=\esssup_{t\in(0,T)}t^\alpha\|f(t)\|_X.
\end{equation*}
It will be future work to sharpen this result and actually obtain the validity of \eqref{eqn:TangentialConeCondition}, which we conjecture to hold.
We will first prove continuity of the operator $P$ (which implies continuity of $F$) and then the above tangential cone condition, which also implies differentiability of $P$ (implying in turn differentiability of $F$).

Let us first introduce the operator $P$.
For given drift $\mu\in C^1(\overline\Omega;\R^d)$ with $\mu\cdot\nu=0$ on $\partial\Omega$, let us define the elliptic differential operator and associated bilinear form
\begin{align*}
L_\mu&:H^1(\Omega)\to H^1(\Omega)^*,&
L_\mu u&=-\Delta(\tfrac{\sigma^2}2u)+\div(\mu u),\\
B_\mu&:H^1(\Omega)^2\to\R,&
B_\mu(u,v)&=\int_\Omega -\mu u \cdot \nabla v + \tfrac{\sigma^2}{2} \nabla u \cdot \nabla v \,\d x.
\end{align*}
Next define $P$ to be the nonlinear operator that maps a drift $\mu\in C^1(\Omega)$ to the function $(t,x,x^0)\mapsto g_\mu(t,x;0,x^0)$ with $u=g_\mu(\cdot,\cdot;0,x^0)$ being the solution of the parabolic partial differential equation (PDE)
\begin{align*}
\partial_tu+L_\mu u&=0
&&\text{in }[0,T]\times\Omega,\\
(\tfrac{\sigma^2}2\nabla u-\mu u)\cdot\nu&=0
&&\text{on }[0,T]\times\partial\Omega,\\
u&=\delta_{x^{0}}
&&\text{at }t=0,
\end{align*}
which in weak form can be expressed as
\begin{align*}
\langle\partial_tu,v\rangle+B_\mu(u,v)&=0
&&\text{for all }t\in(0,T],v\in H^1(\Omega),\\
\lim_{t\searrow0}\langle u,v\rangle&=v(x^0)
&&\text{for all }v\in C^0(\overline\Omega).
\end{align*}
In other words, $P(\mu)$ is the Green function associated with the PDE (due to the PDE's time-invariance it just depends on a single time variable instead of two).
Recall from \cref{sec:forwardOp} that the $i$th component of the forward operator is given by
\begin{equation*}
F(\mu)_i(x^i,x^{i-1})=P(\mu)(t_i-t_{i-1},x^i,x^{i-1}).
\end{equation*}
As a side remark, the requirement $\mu\in C^1(\Omega;\R^d)$ leads to the classical setting for second order parabolic equations,
in which all coefficients of $L_\mu u=-\frac{\sigma^2}2\Delta u+\mu\cdot\nabla u+\div(\mu)u$ are uniformly bounded.
The operator $P$ satisfies the following bounds.

\begin{lemma}[Estimate on fundamental solution, \protect{\cite[Ch.\,1, Sec.\,8]{Friedman}}]\label{thm:EstimatesFundamentalSolution}
For given dimension $d$, $r>1+\frac d2$, and diffusion coefficient $\sigma^2$ there exist a constant $C>0$ and a continuous function $\hat C:[0,\infty)\to(0,\infty)$ such that
	\begin{align*}
		|P(\mu)(t,x,x^0)|
		&\leq\hat C(\|\mu\|_{H^r(\Omega;\R^d)})t^{-\frac{d}2}\ \ \ \exp\big(-C\tfrac{|x-x^0|^2}{t}\big)
		\eqdef G_{0}(t,x-x^0),\\
		|\partial_{x^0} P(\mu)(t,x,x^0)|
		&\leq\hat C(\|\mu\|_{H^r(\Omega;\R^d)})t^{-\frac{d+1}2}\exp\big(-C\tfrac{|x-x^0|^2}{t}\big)
		\eqdef G_{1}(t,x-x^0).
		\end{align*}
\end{lemma}

%\begin{lemma}[Boundedness of Green's function]
%Let $r>\frac d2+1$.
%There exist constants $C,\hat C>0$ only depending on $\Omega$ and $\|\mu\|_{H^r(\Omega;\R^d)}$ such that for all $a,b,c\in\N$ with $2a+b+c<2r$ it holds
%\begin{equation*}
%|\partial_t^a\partial_x^b\partial_{x^0}^cP(\mu)(t,x,x^0)|
%\leq\hat Ct^{-\frac{2a+b+c+d}2}\exp\big(-C\tfrac{|x-x^0|^2}{t}\big)
%\eqdef G_{2a+b+c}(t,x-x^0).
%\end{equation*}
%\todo[inline]{I am guessing the condition on a,b,c. I will focus on the case $c = 0, 2a + b \leq 2$, which corresponds to the coefficients being in $C^{0,\alpha}$ only. Should we state the lemma as a claim about general parabolic PDEs??}
%\end{lemma}

\begin{remark}[Estimates and regularity of Green's function]
For the fundamental solution (instead of Green's function $P(\mu)$) the previous estimate can be found in \cite[Ch.\,1, Sec.\,8]{Friedman}.
For Green's functions $P(\mu)$ we find estimates on $\partial_xP(\mu)$ (instead of $\partial_{x^0}P(\mu)$) in \cite[Ch.\,4, Thm.\,16.3]{Ladyzenskaja}.
However, those estimates can be transferred to estimates of $\partial_{x^0}P(\mu)$ by the fact that for general Green's functions $g_\mu$ it holds $g_\mu(t,x;\tau,x^0) = g_\mu^*(\tau,x^0;t,x)$,
the latter being the Green function of the adjoint PDE. %, for which the estimates on the derivatives in the variable $x_0$ hold by using the same methods. Therefore we obtain the estimate in the lemma from the estimates in \cite[Ch.\,4, Thm.\,16.3]{Ladyzenskaja}
(Note, though, that the estimates in \cite[Ch.\,4, Thm.\,16.3]{Ladyzenskaja} are explicitly stated only for Dirichlet boundary conditions.) %, whereas our problem is of the secon kind). Indeed, since estimating $G^*$ is sufficient and the coefficients of the adjoint equation do not depend on the derivatives of the drift, we need the drift only to be in $C^{0,\alpha}$ (the diffusion in contrast, needs to be in $C^{2,\alpha}$) \protect\cite[Ch.\, 1, Sec.\,8]{Friedman}.
Similar results can also be found in \cite[Sec.\,4, Thm.\,1]{Kalashnikov}. %Notice also that $G_0$ and $G_1$ are both time- and space-invariant under translation.
The estimates are usually derived for drift and reaction coefficients bounded in $C^{0,\lambda}(\overline\Omega)$,
which in our case is satisfied due to $L_\mu u=-\Delta(\frac{\sigma^2}2u)+\mu\cdot\nabla u+(\div\mu)u$ and due to the compact embedding $H^r(\Omega;\R^d)\hookrightarrow C^{1,\lambda}(\overline\Omega)$ for $\lambda<r-1-d/2$.

In fact one expects an estimate
$|\partial_t^a\partial_x^b\partial_{x^0}^cP(\mu)(t,x,x^0)|
\leq\hat C(\|\mu\|_{H^r(\Omega;\R^d)})t^{-\frac{2a+b+c+d}2}\exp\big(-C\tfrac{|x-x^0|^2}{t}\big)$
for any $2a+b+c$ small enough depending on $r$, but this would require some more work.
\end{remark}

As can be readily checked using
\begin{equation*}
\int_{\Omega}\exp\big(-C\tfrac{|x-x^0|^2}{t}\big)\,\d x
\leq\int_{\R^d}\exp\big(-C\tfrac{|x-x^0|^2}{t}\big)\,\d x
=\prod_{k=1}^d\int_{\R}\exp\big(-C\tfrac{|x_k-x^0_k|^2}{t}\big)\,\d x_k
\lesssim t^{d/2},
\end{equation*}
the previous lemma implies the uniform boundedness
\begin{equation}\label{eqn:forwardOperatorBounded}
\|P(\mu)\|_{L_{\alpha}^\infty L^2(\Omega\times\Omega)},
%\|\partial_tP(\mu)\|_{L_{\beta}^\infty L^2(\Omega\times\Omega)}, $\beta=\frac{d+2}4$
\|\partial_{x_0}P(\mu)\|_{L_{\beta}^\infty L^2(\Omega\times\Omega)}
\leq\tilde C
\end{equation}
for some $\tilde C<\infty$ and $\alpha=\frac d4$, $\beta=\frac{d+2}4$, if $\mu$ is bounded in $H^r(\Omega;\R^d)$.
Furthermore, these bounds imply the following estimate on the solution operator in low dimensions.

\begin{lemma}[Boundedness of solution operator]\label{thm:boundedness}
Let $d \leq 3$, %and $\alpha=d/4$.
$\mu \in {H^r(\Omega;\R^d)}$ for $r>\frac d2+1$, and $g\in L_\alpha^\infty L^2(\Omega;\R^d)\cap L_\beta^\infty H^1(\Omega;\R^d)$ % be defined for all $t \in (0,T]$
with $g \cdot \nu = 0$ on $\partial\Omega$. %, where $\nu$ is the normal vector to the boundary of $\Omega$.
Let $u$ solve $\partial_tu+L_\mu u=\div g$ in $\Omega$ with homogeneous initial and Neumann boundary conditions.
There exists $\bar C>0$ only depending on $\|\mu\|_{H^r(\Omega;\R^d)}$ and $\Omega$ such that $\|u\|_{L_0^\infty L^2(\Omega)}\leq\bar C\|g\|_{L_\alpha^\infty L^2(\Omega)}$.
\end{lemma}
\begin{proof}
\todo[inline,disable]{turn below into a proof, do we prefer $\div$ or $\nabla \cdot$ ? Is $w$ defined in this way regular enough to put the norm inside the integral? Is the Langrangian correct with this pairing and can we assume $p \in H^1$?}
%   At this point we need a  fact about stability of the operator $ \left( \partial_t+L_\mu \right)^{-1}$ which is the source-to-solution operator of \eqref{weak form u tilde} and acts from  % $L^\infty_\alpha(H^1(\Omega)^*)$ to $Y^P$ mapping the right-hand side of the equation to its solution and is linear if we assume trivial initial conditions. Consider the following estimate:
% 	\begin{align} \label{stability estimate}
% 		\left\| \Tilde{u} \right\|_{L^\infty_\alpha(L^2(\Omega))}
% 		\lesssim \left\| \left( \partial_t+L_\mu \right) \Tilde{u} \right\|_{L^\infty_\alpha(H^1(\Omega)^*)}.
% 	\end{align}
%	Let us see that this estimate holds in low dimensions. In fact, since $\Tilde{u}$ solves \eqref{weak form u tilde} we know $\left( \partial_t+L_\mu \right) \Tilde{u} = \div (h\  P(\mu+h)) = \div % (hu_{\mu +h}) \eqdef \div g$. Recall the fact that the drift is assumed to have vanishing normal component at the boundary and thus $g \cdot \nu$ on $\partial\Omega$ as well.
	By Duhamel's principle,
	$
		u(t,x) = \int_0^t \int_\Omega P(\mu)(t-\tau,x,x^0) \div_{x^0} g (\tau,x^0) \,\d x^0 \, \d\tau
	$
	so that
	\begin{align} \label{u tilde has bounded L^2 norm}
		\| u(t,\cdot)\|_{L^2(\Omega)}
		&=  \left\| \int_0^t \int_\Omega  P(\mu)(t-\tau,\cdot,x^0) \div_{x^0} g (\tau,x^0)  \,\d x^0 \ \d\tau \right\|_{L^2(\Omega)} \nonumber\\
		&=  \left\| \int_0^t \int_\Omega \partial_{x^0} P(\mu)(t-\tau,\cdot,x^0) g (\tau,x^0)  \,\d x^0 \ \d\tau \right\|_{L^2(\Omega)} \nonumber\\
		&\leq  \left\| \int_0^t \int_\Omega | G_1(t-\tau,\cdot-x^0) | | g (\tau,x^0) | \,\d x^0 \ \d\tau \right\|_{L^2(\Omega)} \nonumber\\
		&=  \left\| \int_0^t | G_1(t-\tau,\cdot)| \ast | g (\tau,\cdot) | \ \d\tau \right\|_{L^2(\Omega)} \nonumber\\
		&\leq   \int_0^t  \left\| |G_1(t-\tau,\cdot)| \ast |g (\tau,\cdot)| \right\|_{L^2(\Omega)} \ \d\tau \nonumber \\
		&\leq    \int_0^t  \left\| G_1(t-\tau,\cdot) \right\|_{L^1(\R^d)} \left\|  g (\tau,\cdot) \right\|_{L^2(\Omega;\R^d)} \ \d\tau \nonumber \\
% 		&\leq   \int_0^t \frac{\tau^\alpha}{\tau^\alpha} \left\|G_1(t-\tau,\cdot)
% 		 \right\|_{L^1(\R^d)} \left\|  g (\tau,\cdot) \right\|_{L^2(\Omega)} \d\tau \nonumber\\
		&\leq \left\|  g\right\|_{L^\infty_\alpha L^2(\Omega;\R^d)} \int_0^T \tau^{-\alpha} \left\|G_1(t-\tau,\cdot) \right\|_{L^1(\R^d)} \ \d\tau\nonumber\\
    &=\hat C(\|\mu\|_{H^r(\Omega;\R^d)})\left(\frac{2\pi}C\right)^{d/2}\left\|  g \right\|_{L^\infty_\alpha L^2(\Omega;\R^d)}\int_0^t \tau^{-\alpha} \frac1{\sqrt{t-\tau}}\,\d\tau\nonumber\\
    &=\tilde C\|  g \|_{L^\infty_\alpha L^2(\Omega;\R^d)},
	\end{align}
	where we used integration by parts, lemma \ref{thm:EstimatesFundamentalSolution}, the translation invariance of $G_1$, Young's convolution inequality, and $\alpha=\frac d4<1$.
% 	We proceed using that $d \leq 3$ and that $\alpha = \frac{d}{4}$, then again by lemma \ref{thm:EstimatesFundamentalSolution} we get
% 	\begin{align}
% 		\left\|G_1(t-s,\cdot) \right\|_{L^1(\Omega)} \leq \frac{\bar{C}}{(t-s)^{\frac{1}2}},
% 	\end{align}
% 	where the constant in the estimate depends on the drift $\mu$ as in the lemma and thus
% 	\begin{align} \label{L2 norm u tilde}
% 		\| u(t,\cdot)\|_{L^2(\Omega)}
% 		&\leq \bar{C}  \underbrace{ \int_0^t (t-s)^{-\frac{1}{2}} s^{-\frac{d}{4}} \ ds}_{< \infty} \left\|  g  \right\|_{L_\alpha^\infty (L^2(\Omega))} \nonumber\\
% 		&\leq \bar{C} \left\|  g  \right\|_{L_\alpha^\infty L^2(\Omega)}.
% 	\end{align}
% 	Multiplying both sides by $t^\alpha$ and taking the supremum between $0$ and $T$ essentially only affects the left side and yields
% 	\begin{align} \label{u tilde against g}
% 		\| u\|_{L^\infty_\alpha L^2(\Omega)} \leq \bar{C} \left\|  g  \right\|_{L_\alpha^\infty L^2(\Omega)}.
% 	\end{align}
\todo[inline,disable]{Following seems not needed:\\
Next note that in the above we can replace $g(\tau,\cdot)$ with any vector field $w(\tau,\cdot)$ having the same divergence as $g(\tau,\cdot)$,
in particular with $w(\tau,\cdot)=\nabla p$ for $p\in H^1(\Omega)$ the solution of $\Delta p=\div g(\tau,\cdot)$ with homogeneous Neumann boundary conditions (and vanishing average).
By the Lax--Milgram lemma there exists some $c>0$ depending only on $\Omega$ with
\begin{equation*}
\|w(\tau,\cdot)\|_{L^2(\Omega)}
=\|\nabla p\|_{L^2(\Omega)}
\leq\|p\|_{H^1(\Omega)}
\leq c\|\div g(\tau,\cdot)\|_{H^1(\Omega)^*},
\end{equation*}
thus in summary we arrive at $\|u(t,\cdot)\|_{L^2(\Omega)}\leq\tilde C\|w\|_{L^\infty_\alpha L^2(\Omega)}\leq\tilde Cc\|\div g(\tau,\cdot)\|_{L_\alpha(H^1(\Omega)^*)}$, as desired.
}
\end{proof}

We next show that operator $P$ is continuous with respect to the $L^\infty$-topology.

\begin{lemma}[Continuity of Green function operator]\label{thm:continuityForwardOp}
% 	The operator $P: (H^r(\Omega)^d \cap L^\infty(\Omega)^d \cap H^1_\nu(\Omega)^d,\|\cdot\|_{L^\infty}) \longrightarrow (Y^P,\| \cdot \|_{L^\infty_{\alpha}(L^2(\Omega))})$ is continuous for $d \leq 3$, $\alpha = \frac{d}{4}$ and $r > \frac{d}{2} + 1$.
Let $d \leq 3$ and $\mu \in H^r_{\mathrm N}$ for $r>\frac d2+1$.
There exists $\tilde C>0$ depending only on $\|\mu\|_{H^r(\Omega;\R^d)}$ such that
\begin{equation*}
(1-\tilde C\|h\|_{L^\infty})\|P(\mu+h)-P(\mu)\|_{L_\alpha^\infty L^2(\Omega\times\Omega)}
\leq\tilde C\|h\|_{L^\infty}\|P(\mu)\|_{L_\alpha^\infty L^2(\Omega\times\Omega)}
\qquad\text{for all }h\in H^r_{\mathrm N}.
\end{equation*}
\end{lemma}
\begin{proof}
For fixed $x^0$ define $u_h(t,x)\defeq P(\mu+h)(t,x,x^0)$ and $\tilde u_\mu^h(t,x)\defeq[P(\mu+h)-P(\mu)](t,x,x^0)$, then $\tilde u_\mu^h$ satisfies
\begin{equation*}
\partial_t\tilde u_\mu^h+L_\mu\tilde u_\mu^h=-\div(hu_h).
\end{equation*}
with homogeneous initial and Neumann boundary conditions.
Due to $u_h\in L_\alpha^\infty L^2(\Omega)\cap L_\beta^\infty H^1(\Omega)$ by \eqref{eqn:forwardOperatorBounded} and $r>1+\frac d2$ we can apply lemma \ref{thm:boundedness} with $g=-hu_h$ to obtain
\begin{equation*}
\|\tilde u_\mu^h\|_{L_0^\infty L^2(\Omega)}\leq\bar C\|hu_h\|_{L_\alpha^\infty L^2(\Omega)}\leq\bar C\|h\|_{L^\infty(\Omega)}\|u_h\|_{L_\alpha^\infty L^2(\Omega)},
\end{equation*}
which for $\tilde C=\bar CT^\alpha$ implies the desired inequality
\begin{multline*}
\|P(\mu+h)-P(\mu)\|_{L_\alpha^\infty L^2(\Omega\times\Omega)}
\leq\tilde C\|h\|_{L^\infty(\Omega)}\|P(\mu+h)\|_{L_\alpha^\infty L^2(\Omega\times\Omega)}\\
\leq\tilde C\|h\|_{L^\infty(\Omega)}[\|P(\mu)\|_{L_\alpha^\infty L^2(\Omega\times\Omega)}+\|P(\mu+h)-P(\mu)\|_{L_\alpha^\infty L^2(\Omega\times\Omega)}].
\qedhere
\end{multline*}
\end{proof}

We are finally in the position to derive differentiability of $P$ with respect to the $L^\infty$-topology, which implies the desired tangential cone condition.

\begin{proposition}[Differentiability and tangential cone condition for Green function operator]\label{thm:weakTangentialConeCondition}
Let $d \leq 3$ and $\mu \in H^r_{\mathrm N}$ for $r>\frac d2+1$.
Define the operator $P'(\mu):H^r_{\mathrm N}\to L_\alpha^\infty L^2(\Omega\times\Omega)\cap L_\beta^\infty H^1(\Omega\times\Omega)$ via
\begin{equation*}
P'(\mu)[h](t,x,x^0)=u_\mu^h(t,x)
\qquad\text{for $u_\mu^h$ the solution of}\qquad
\partial_tu_\mu^h+L_\mu u_\mu^h=-\div(hP(\mu)(\cdot,\cdot,x^0))
\end{equation*}
with homogeneous initial and Neumann boundary conditions,
then the tangential cone condition \eqref{eqn:weakTangentialConeCondition} (whose right-hand side is $o(\|h\|_{L^\infty(\Omega)})$ due to lemma \ref{thm:continuityForwardOp}) holds for $\mu^\dagger=\mu$.
\end{proposition}
\begin{proof}
For fixed $x^0$ let again $\tilde u_\mu^h(t,x)=[P(\mu+h)-P(\mu)](t,x,x^0)$, then $\tilde u(t,x)=[P(\mu+h)-P(\mu)-P'(\mu)[h]](t,x,x^0)$ satisfies
\begin{equation*}
\partial_t\tilde u+L_\mu\tilde u=-\div(h\tilde u_\mu^h)
\end{equation*}
with homogeneous initial and Neumann boundary conditions.
Again appealing to lemma \ref{thm:boundedness} (which is possible due to $u_\mu^h\in L_\alpha^\infty L^2(\Omega)\cap L_\beta^\infty H^1(\Omega)$ by \eqref{eqn:forwardOperatorBounded} and $r>1+\frac d2$) we obtain
\begin{equation*}
T^{-\alpha}\|\tilde u\|_{L_\alpha^\infty L^2(\Omega)}
\leq\|\tilde u\|_{L_0^\infty L^2(\Omega)}
\leq\bar C\|h\tilde u_\mu^h\|_{L_\alpha^\infty L^2(\Omega)}
\leq\bar C\|h\|_{L^\infty}\|\tilde u_\mu^h\|_{L_\alpha^\infty L^2(\Omega)},
\end{equation*}
which (upon squaring and integrating in $x^0$) directly implies the desired inequality.
\end{proof}

\section{Numerical experiments}\label{sec:numerics}
In the following, we perform numerical experiments to empirically validate the existence of a convergence rate in \cref{convergence rates Dunker Expectation} -- in fact, we will observe a rate $O(n^{-1/2})$.
For this we perform a series of simulations that model motion of $n$ particles in a one-dimensional domain $\Omega$ over time $t \in [0,1]$.
At the beginning of the simulation for time $t=0$ the particles are uniformly distributed in $[0,1] \subset \Omega$.
Motion of each particle is (independently) governed by the SDE \eqref{eqn:SDE}, in which we express the drift $\mu$ as the superposition of a known spatially constant flux $u(x)=5$ and the gradient of a potential $\Phi\colon \Omega \rightarrow \R$,
\begin{equation*}
\mu=u+\nabla\Phi
\end{equation*}
(this is allowed since in one space dimension there is a one-to-one relation between potentials and drifts).
% Here, $x_i \colon [0,1] \rightarrow \Omega$ is the position of particle $i \in \lbrace1,\ldots,n\rbrace$ at time point $t \in [0,1]$, $u \colon [0,1] \times \Omega \rightarrow \R$ is a drift term that can vary over time and space.
% The function $\Phi \colon [0,1] \times \Omega \rightarrow \R$ is a potential that can attract or repel particles in $\Omega$, while $W_i \colon [0,1] \rightarrow \R$ models the stochasticity of each particle via Brownian motion.
Our goal in this numerical experiment is to infer the potential $\Phi$ from the observation of motion trajectories of the $n$ simulated particles.

We use the Euler--Maruyama (EM) method with $M$ time steps to discretize the SDE in time and track the position of the particles in our simulations.
For simplicity, in our simulations we assume that the observation times $t_0,\ldots,t_M$ coincide with these discretized time steps, \ie we have a large number of observation times.
The EM discretisation provides a direct way of computing the likelihood from one time step to the next.
Since we can compute the position $q_j$ of particle $j \in \lbrace 1,\ldots,n \rbrace$ for the next time step $k+1$ via
\begin{align*}
    q_j^{(i+1)} = q^{(i)}_j %+ \Delta t \alpha u(t, q^{(i)}_j)
    + \Delta t\mu(q^{(i)}_j)+ \sigma \sqrt{\Delta t}\xi,
\end{align*}
with $\xi\sim\mathcal N(0,1)$ normally distributed,
we know that for fixed $\sigma$ and potential function $\Phi$ it holds
\begin{align*}
    \left.q_j^{(i+1)}\right| q^{(i)}_j \ \sim \ \mathcal N\left( q^{(i)}_j %+ \Delta t \alpha u(t, q^{(i)}_j)
    + \Delta t\mu(q^{(i)}_j), \: \sigma^2\Delta t \right).
\end{align*}
As described in the previous sections, our parameter estimation is based on the fidelity term $S_\tau$, the log-likelihood of the observations, shifted by some $\tau>0$.
In practice, we actually use $\tau=0$, \ie we take $S$ as our fidelity term.
The likelihood of the discrete motion trajectories can practically be calculated by multiplying the above densities for each time step sequentially, using the Markovian property of the underlying SDE and its discretisation.
In practice, we directly use the log-likelihood (and summation) to avoid numerical underflow and enable us to evaluate the likelihood for a fixed choice of parameters and a given potential.
This approach has the benefit that these parameters and functions can be used within a standard optimisation procedure to infer the potential $\Phi$ from the observed trajectories.

As in the previous sections, we assume $\sigma$ to be known.
% Since we are mainly interested in empirically validation of the convergence rate stated in this paper, we assume the parameters $\alpha$ and $\beta$ to be known.
% For simplicity we fix the drift term $u$ to be an external constant flow from left to right, leaving only the problem of inferring the potential $\Phi$.
As ground truth $\Phi^\dagger$ of the potential we choose a simple double-well potential, discretised in Fourier space.
Naturally, we perform optimisation during inference of $\Phi$ in the space of Fourier modes.
Note that in this numerical setup we do not need any additional regularisation since this simple one-dimensional problem is relatively well-posed.

In \cref{fig:numerics} we present the results of two particular numerical experiments simulating $n=12$ (left) and $n=120$ particles.
In the top row we show simulated trajectories for the uniformly distributed particles.
The $x$-axis shows the respective position $q_j \in \Omega$ of the $n$ particles, while the $y$-axis denotes the time $t \in [0,1]$.
Note that for the case $n=120$ (right) we used shades of grey to accumulate trajectories and highlighted only a few particular trajectories in colour for a better visualization.
In the second row we visualize the inferred and true potentials for both numerical experiments.
Peaks in the potential lead to repulsion of particles, while valleys in the potential will attract particles.
This can already be observed in the simulated trajectories in the top row.

While for the case $n=12$ (left) the inferred potential (orange curve) loosely follows the characteristics of the ground truth potential (blue curve), the inferred potential in the case $n=120$ (right) is much closer to the true potential.
\begin{figure}
    \centering
    \includegraphics[width=0.5\linewidth]{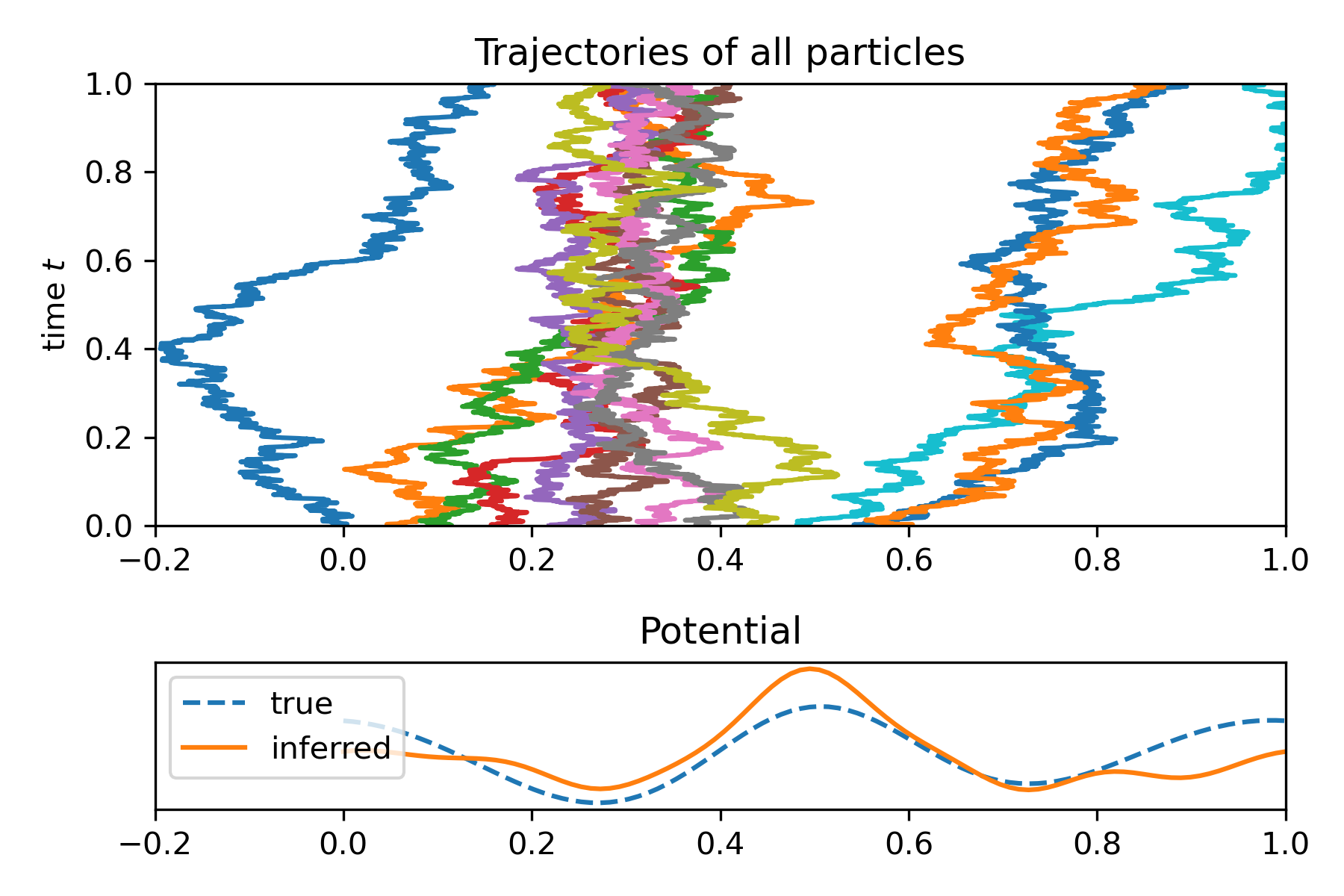}%
    \includegraphics[width=0.5\linewidth]{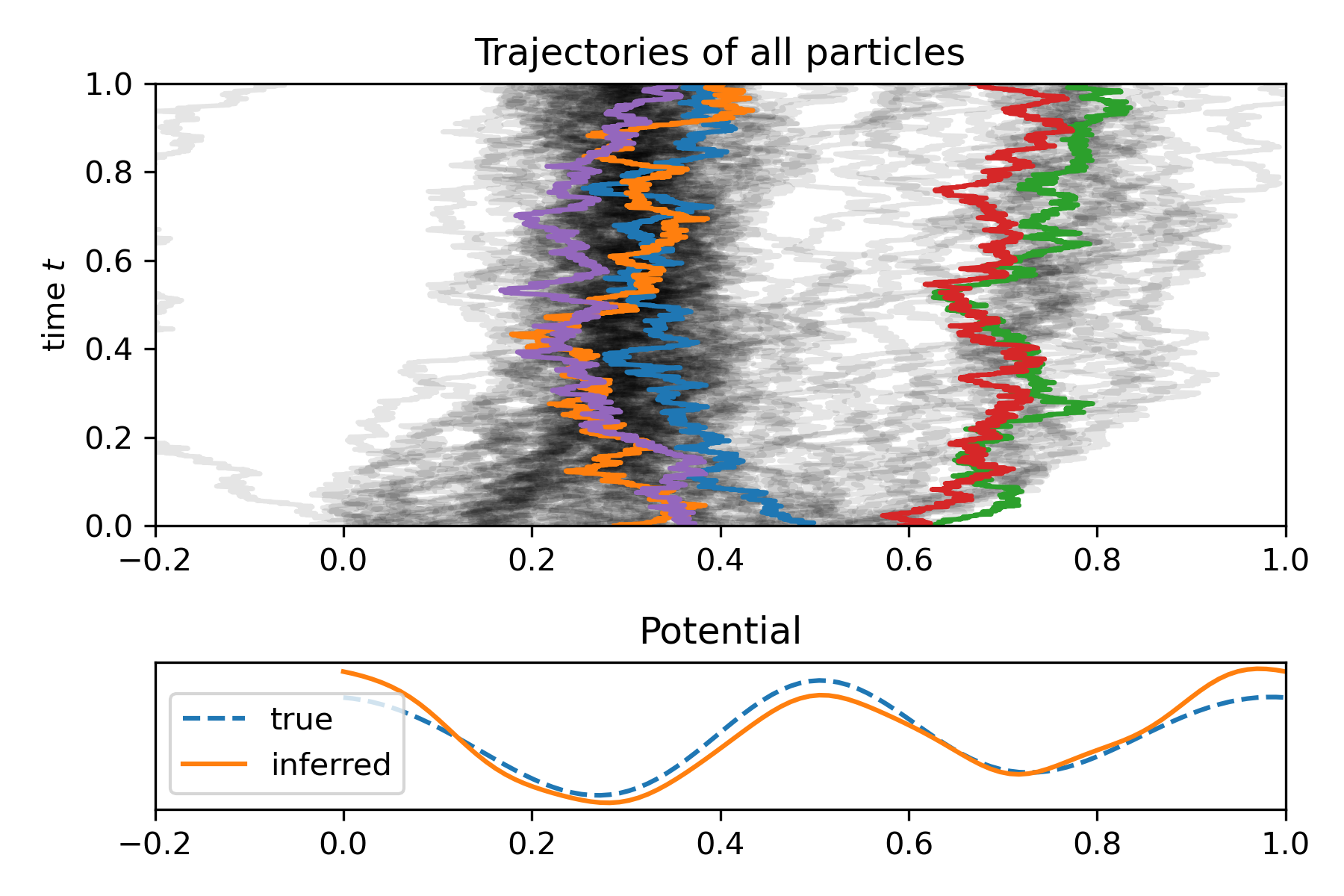}\\
    \caption{Left: Simulated trajectories (top) and potential inference (bottom) for $n=12$ particles. Right: Simulated trajectories (top) and potential inference (bottom) for $n=120$ particles. }
    \label{fig:numerics}
\end{figure}

To quantify the error between the inferred potential and the true underlying potential, we performed additional experiments in which the distance is measured in the $L^2$ norm.
In particular, we infer the potential $\Phi$ from the simulated particle trajectories for an increasing amount of particles $n=2^k$, for $k=3,\ldots,10$.
To take the stochasticity of the simulations into account, we perform $25$ independent experiments for each amount $n$ of particles and subsequently compute the mean value and variance of the computed $L^2$ errors.
\Cref{fig:convergence_rate} shows a box plot of the computed $L^2$ errors on a logarithmic scale.
Additionally, we plot two reference lines for different theoretical rates of convergence: the blue line shows a convergence rate of $\mathcal O(n^{-1/2})$, while the orange curve shows a convergence rate of $\mathcal O(n^{-1})$.
\begin{figure}
\centering
\setlength\unitlength\linewidth
\begin{picture}(.72,.33)
\put(.02,.02){\includegraphics[width=0.7\unitlength,trim=30 30 10 10,clip]{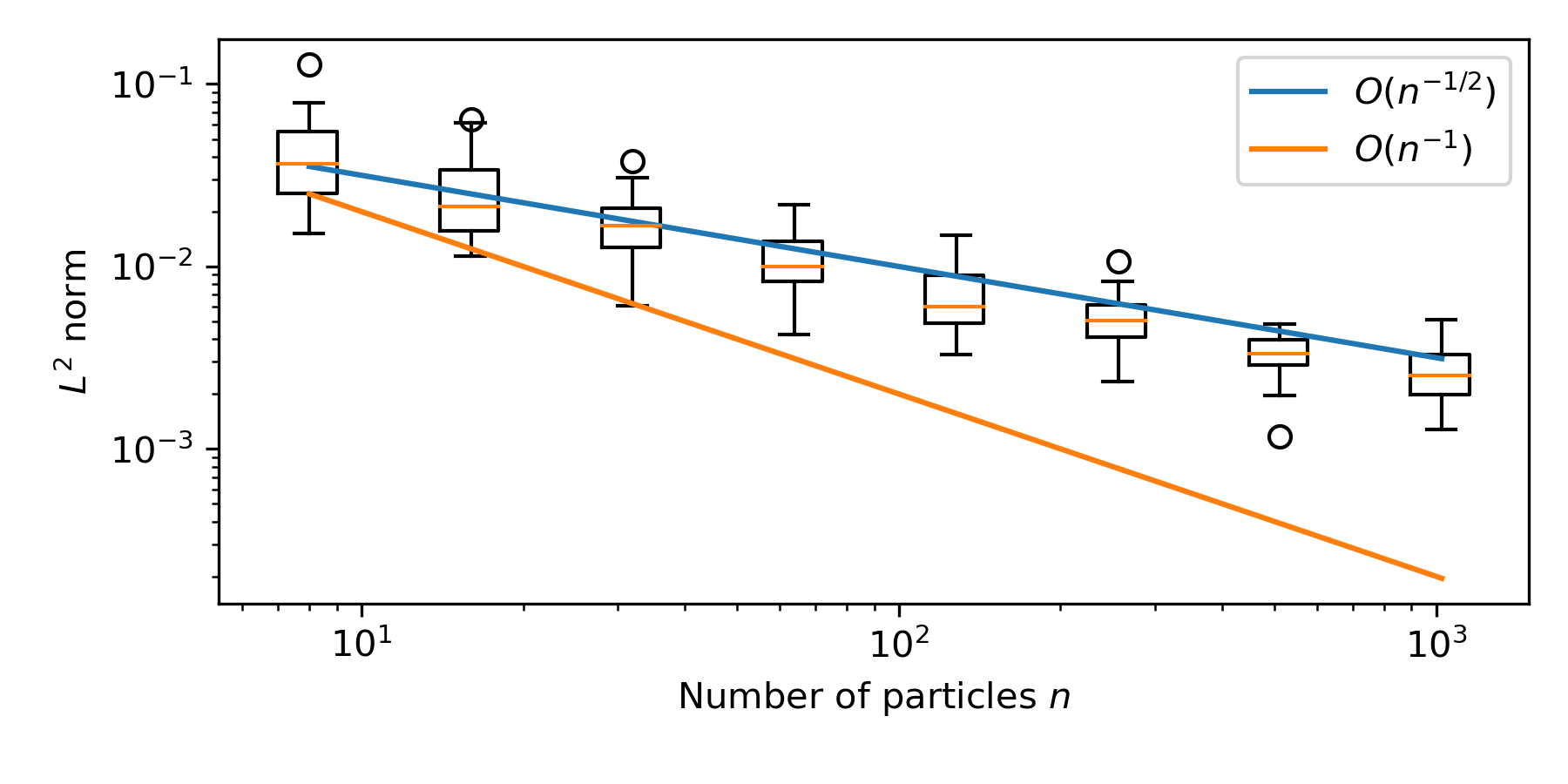}}
\put(.32,0){Number $n$ of particles}
\put(0,.13){\rotatebox{90}{$L^2$-norm error}}
\put(.63,.265){\color{white}\rule{8.5ex}{6ex}}
\put(.625,.288){\begin{minipage}{10ex}$O(n^{-1/2})$\\[0.1\baselineskip]$O(n^{-1})$\end{minipage}}
\end{picture}
  \caption{Error between ground truth and inferred potential in $L^2$ norm, as a function of number of particles tracked, visualised as a box plot over $25$ independent runs each. Asymptotic curves for $\mathcal O(n^{-1/2})$ and $\mathcal O(n^{-1})$ show that convergence is close to $\mathcal O(n^{-1/2})$.}
  \label{fig:convergence_rate}
\end{figure}
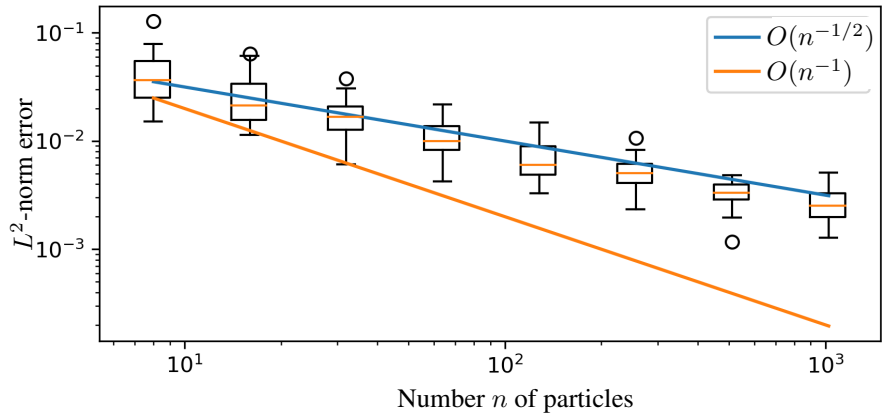
As can be observed, our numerical experiments exhibit a convergence rate of $\mathcal O(n^{-1/2})$.

    For our numerical experiments we worked in finite dimensions,
	\ie the (smooth) probability distribution of particles is governed by the only finitely many Fourier coefficients of the potential $\Phi$
	through the nonlinear, yet in finite dimensions smooth and injective forward operator.
	Essentially this also explains why numerically we did not require any regularization.
	In such a setting one expects an $L^2$-convergence rate of $\mathcal O(n^{-1/2})$ for approximating the (nonlinearly, but finite-dimensionally) parameterized probability distribution from $n$ empirical samples
	(and thus automatically the same rate for approximating the potential).
	In contrast, convergence \cref{convergence rates Dunker Expectation} for the truly infinite-dimensional setting does not come with an explicit rate.
	In fact, the rate in general depends on the smoothness of the ground truth,
	so an obvious question is whether our experimentally observed rate is really due to the regularity of the ground truth potential or rather an artefact of working with a finite-dimensional parameterization of the potential.
	In order to avoid the latter, above we employed more Fourier coefficients than actually necessary to represent the ground truth potential;
	however, a definitive answer would require more numerical testing of how the observed rate depends on the number of Fourier coefficients or on the ground truth regularity.
	
	Coming back to the application, the biological data from \cref{fig:embryo} consists of almost a thousand embryo recordings with each a handful of PGCs, accumulating to an order of $n=10,000$ observed cells.
	One could argue that, assuming a rate of $\mathcal O(n^{-1/2})$, this should yield a relative $L^2$-norm error two orders of magnitude smaller than the relative error obtained from observing a single cell,
	which one expects to be of order one.
	Given all other modelling uncertainties this accuracy seems sufficient.
	In fact, looking at \cref{fig:embryo}, one expects to be able to extract at least some useful information since already the final cell distribution looks quite characteristic.
	The obtainable spatial resolution, though, is difficult to predict and also depends on the amplitude of the ground truth potential:
	Indeed, from large deviation theory one knows that the time for a particle to move between two wells is exponential in the \emph{potential barrier} divided by $\sigma^2$.
	
	Our theoretical result \cref{thm:weakTangentialConeCondition} is as yet insufficient to obtain a convergence rate for the rconstruction of the potential from observing particles at discrete time points:
	Inside the norms, the supremum over all times would have to be replaced by evaluation at specific discrete time points.
	Since the norms are taken of solutions to parabolic equations, whose time-reversal is ill-posed, the time-dependent norm can unfortunately not be bounded in terms of the norm at the discrete time points.
	However, the structure of all involved estimates is to estimate the norm of the solution of the parabolic equation in terms of the norm of its right-hand side, which in turn is the solution of a closely related parabolic equation.
	It seems plausible that such estimates also hold at specific times (rather than in time-dependent norms), since the PDE solution and the PDE right-hand side smoothen at the same rate.
	In essence we expect that one needs a quantitative version of classical uniqueness results for identifying a drift or an initial condition from the PDE solution at a discrete time point.

\section*{acknowledgement}
This work was supported by the Deutsche Forschungsgemeinschaft (DFG, German Research Foundation)
via project 431460824 -- Collaborative Research Center 1450
and via Germany's Excellence Strategy project 390685587 -- Mathematics M\"unster: Dynamics-Geometry-Structure.

\vspace{\baselineskip}
%% The style of the following references should be used in all documents.

\bibliographystyle{plain}  
\bibliography{Bibfile}

@article{Seidman-Vogel,
	AUTHOR = {Seidman, Thomas I. and Vogel, Curtis R.},
	TITLE = {Well-posedness and convergence of some regularisation methods
	for nonlinear ill posed problems},
	JOURNAL = {Inverse Problems},
	FJOURNAL = {Inverse Problems. An International Journal on the Theory and
	Practice of Inverse Problems, Inverse Methods and Computerized
	Inversion of Data},
	VOLUME = {5},
	YEAR = {1989},
	NUMBER = {2},
	PAGES = {227--238},
	ISSN = {0266-5611,1361-6420},
	MRCLASS = {65J15 (45G10 47H17)},
	MRNUMBER = {991919},
	MRREVIEWER = {Heinz\ W.\ Engl},
	DOI = {10.1088/0266-5611/5/2/008},
	URL = {https://doi.org/10.1088/0266-5611/5/2/008},
}

@article{Engl-Kunisch-Neubauer,
    AUTHOR = {Engl, Heinz W. and Kunisch, Karl and Neubauer, Andreas},
     TITLE = {Convergence rates for {T}ikhonov regularisation of nonlinear
              ill-posed problems},
   JOURNAL = {Inverse Problems},
  FJOURNAL = {Inverse Problems. An International Journal on the Theory and
              Practice of Inverse Problems, Inverse Methods and Computerized
              Inversion of Data},
    VOLUME = {5},
      YEAR = {1989},
    NUMBER = {4},
     PAGES = {523--540},
      ISSN = {0266-5611,1361-6420},
   MRCLASS = {65J15 (47H17)},
  MRNUMBER = {1009037},
MRREVIEWER = {Charles\ W.\ Groetsch},
       DOI = {10.1088/0266-5611/5/4/007},
       URL = {https://doi.org/10.1088/0266-5611/5/4/007},
}

@article{Burger-Osher,
    AUTHOR = {Burger, Martin and Osher, Stanley},
     TITLE = {Convergence rates of convex variational regularization},
   JOURNAL = {Inverse Problems},
  FJOURNAL = {Inverse Problems. An International Journal on the Theory and
              Practice of Inverse Problems, Inverse Methods and Computerized
              Inversion of Data},
    VOLUME = {20},
      YEAR = {2004},
    NUMBER = {5},
     PAGES = {1411--1421},
      ISSN = {0266-5611,1361-6420},
   MRCLASS = {49N45 (35J20 35J60)},
  MRNUMBER = {2109126},
MRREVIEWER = {Baisheng\ Yan},
       DOI = {10.1088/0266-5611/20/5/005},
       URL = {https://doi.org/10.1088/0266-5611/20/5/005},
}

@article{Resmerita-Scherzer,
    AUTHOR = {Resmerita, Elena and Scherzer, Otmar},
     TITLE = {Error estimates for non-quadratic regularization and the
              relation to enhancement},
   JOURNAL = {Inverse Problems},
  FJOURNAL = {Inverse Problems. An International Journal on the Theory and
              Practice of Inverse Problems, Inverse Methods and Computerized
              Inversion of Data},
    VOLUME = {22},
      YEAR = {2006},
    NUMBER = {3},
     PAGES = {801--814},
      ISSN = {0266-5611,1361-6420},
   MRCLASS = {65J20 (47J06 47J25)},
  MRNUMBER = {2235638},
       DOI = {10.1088/0266-5611/22/3/004},
       URL = {https://doi.org/10.1088/0266-5611/22/3/004},
}

@article{Hofmann-Kaltenbacher-Poschl-Scherzer,
    AUTHOR = {Hofmann, B. and Kaltenbacher, B. and P\"oschl, C. and
              Scherzer, O.},
     TITLE = {A convergence rates result for {T}ikhonov regularization in
              {B}anach spaces with non-smooth operators},
   JOURNAL = {Inverse Problems},
  FJOURNAL = {Inverse Problems. An International Journal on the Theory and
              Practice of Inverse Problems, Inverse Methods and Computerized
              Inversion of Data},
    VOLUME = {23},
      YEAR = {2007},
    NUMBER = {3},
     PAGES = {987--1010},
      ISSN = {0266-5611,1361-6420},
   MRCLASS = {65J20 (47J06 47J25)},
  MRNUMBER = {2329928},
MRREVIEWER = {Mihail\ Yu.\ Kokurin},
       DOI = {10.1088/0266-5611/23/3/009},
       URL = {https://doi.org/10.1088/0266-5611/23/3/009},
}

@article{Hofmann-Yamamoto,
    AUTHOR = {Hofmann, Bernd and Yamamoto, Masahiro},
     TITLE = {On the interplay of source conditions and variational
              inequalities for nonlinear ill-posed problems},
   JOURNAL = {Appl. Anal.},
  FJOURNAL = {Applicable Analysis. An International Journal},
    VOLUME = {89},
      YEAR = {2010},
    NUMBER = {11},
     PAGES = {1705--1727},
      ISSN = {0003-6811,1563-504X},
   MRCLASS = {47J06 (35R30 47J20 65J20)},
  MRNUMBER = {2683677},
MRREVIEWER = {Mihail\ Yu.\ Kokurin},
       DOI = {10.1080/00036810903208148},
       URL = {https://doi.org/10.1080/00036810903208148},
}

@article{Hofmann-Bot,
    AUTHOR = {Bo\c{t}, Radu Ioan and Hofmann, Bernd},
     TITLE = {An extension of the variational inequality approach for
              obtaining convergence rates in regularization of nonlinear
              ill-posed problems},
   JOURNAL = {J. Integral Equations Appl.},
  FJOURNAL = {Journal of Integral Equations and Applications},
    VOLUME = {22},
      YEAR = {2010},
    NUMBER = {3},
     PAGES = {369--392},
      ISSN = {0897-3962,1938-2626},
   MRCLASS = {47J06 (47J20 49J40 65J20 65K15)},
  MRNUMBER = {2727323},
MRREVIEWER = {Markus\ Haltmeier},
       DOI = {10.1216/JIE-2010-22-3-369},
       URL = {https://doi.org/10.1216/JIE-2010-22-3-369},
}

@phdthesis{Poschl,
author = {Pöschl, Christiane},
year = {2008},
month = {10},
pages = {},
title = {Tikhonov Regularization with General Residual Term}
}

@book{Flemming,
author = {Jens Flemming},
title = {Generalized Tikhonov regularization and modern convergence rate theory in Banach spaces},
series  = {Berichte aus der Mathematik},
publisher = {Shaker Verlag},
year = {2012},
address = {Aachen},
isbn = {9783844008678}
}

@phdthesis{Werner,
  author       = {Frank Werner},
  title        = {Inverse Problems with Poisson Data: Tikhonov-Type Regularization and Iteratively Regularized Newton Methods},
  school       = {University of G\"ottingen},
  year         = {2012},
  address      = {G\"ottingen, Germany},
  note         = {Doctoral dissertation},
  isbn         = {978-3-86247-239-0},
  url          = {https://www.deutsche-digitale-bibliothek.de/item/UBZTKOEGEF64GAM6NJRPJ2QBYAKVXQ5L}
}

@article{Werner-Hohage,
    AUTHOR = {Werner, Frank and Hohage, Thorsten},
     TITLE = {Convergence rates in expectation for {T}ikhonov-type
              regularization of inverse problems with {P}oisson data},
   JOURNAL = {Inverse Problems},
  FJOURNAL = {Inverse Problems. An International Journal on the Theory and
              Practice of Inverse Problems, Inverse Methods and Computerized
              Inversion of Data},
    VOLUME = {28},
      YEAR = {2012},
    NUMBER = {10},
     PAGES = {104004, 15},
      ISSN = {0266-5611,1361-6420},
   MRCLASS = {65J22 (65J15 65J20)},
  MRNUMBER = {2987899},
MRREVIEWER = {Anton\ Suhadolc},
       DOI = {10.1088/0266-5611/28/10/104004},
       URL = {https://doi.org/10.1088/0266-5611/28/10/104004},
}

@article{Dunker-Hohage,
    AUTHOR = {Dunker, Fabian and Hohage, Thorsten},
     TITLE = {On parameter identification in stochastic differential
              equations by penalized maximum likelihood},
   JOURNAL = {Inverse Problems},
  FJOURNAL = {Inverse Problems. An International Journal on the Theory and
              Practice of Inverse Problems, Inverse Methods and Computerized
              Inversion of Data},
    VOLUME = {30},
      YEAR = {2014},
    NUMBER = {9},
     PAGES = {095001, 20},
      ISSN = {0266-5611,1361-6420},
   MRCLASS = {62G05 (34F05 60H10 65C30 65L09)},
  MRNUMBER = {3257994},
MRREVIEWER = {Truc\ Nguyen},
       DOI = {10.1088/0266-5611/30/9/095001},
       URL = {https://doi.org/10.1088/0266-5611/30/9/095001},
}

@article{Borwein,
    AUTHOR = {Borwein, J. M. and Lewis, A. S.},
     TITLE = {Convergence of best entropy estimates},
   JOURNAL = {SIAM J. Optim.},
  FJOURNAL = {SIAM Journal on Optimization},
    VOLUME = {1},
      YEAR = {1991},
    NUMBER = {2},
     PAGES = {191--205},
      ISSN = {1052-6234},
   MRCLASS = {41A46 (28A20 28D20)},
  MRNUMBER = {1098426},
MRREVIEWER = {Tiberiu\ Constantinescu},
       DOI = {10.1137/0801014},
       URL = {https://doi.org/10.1137/0801014},
}

@book{Ladyzenskaja,
    AUTHOR = {Lady\v{z}enskaja, O. A. and Solonnikov, V. A. and Ural\cprime{}ceva, N. N.},
     TITLE = {Linear and quasilinear equations of parabolic type},
    SERIES = {Translations of Mathematical Monographs},
    VOLUME = {23},
      NOTE = {Translated from the Russian by S. Smith},
 PUBLISHER = {American Mathematical Society, Providence, RI},
      YEAR = {1968},
     PAGES = {xi+648},
   MRCLASS = {35.62},
  MRNUMBER = {241822},
MRREVIEWER = {B.\ Frank\ Jones, Jr.},
}

@article{Kalashnikov,
    AUTHOR = {Il\cprime{}in, A. M. and Kala\v{s}nikov, A. S. and Ole\u{i}nik, O. A.},
  title     = {Linear equations of the second order of parabolic type},
  journal   = {Uspekhi Matematicheskikh Nauk},
  volume    = {17},
  number    = {3},
  pages     = {3--146},
  year      = {1962},
  doi       = {10.1070/RM1962v017n03ABEH004115},
  url       = {https://doi.org/10.1070/RM1962v017n03ABEH004115}
}

@article{Gross-Thebing,
  author = {S. Gross-Thebing and L. Truszkowski and D. Tenbrinck and H. Sanchez-Iranzo and C. Camelo and K.J. Westerich and A. Singh and P. Maier and J. Prengel and P. Lange and J. Huewel and F. Gaede and R. Sasse and B.E. Vos and T. Betz and M. Matis and R. Prevedel and S. Luschnig and A. Diz-Munoz and M. Burger and E. Raz},
  title = {Using migrating cells as probes to illuminate features in live embryonic tissues},
  journal = {Science Advances},
  volume = {6},
  number = {49},
  year = {2020},
  doi = {10.1126/sciadv.abc5546},
  url = {https://www.science.org/doi/10.1126/sciadv.abc5546}
}

@article{mathe,
    AUTHOR = {Math\'e, Peter and Hofmann, Bernd},
     TITLE = {How general are general source conditions?},
   JOURNAL = {Inverse Problems},
  FJOURNAL = {Inverse Problems. An International Journal on the Theory and
              Practice of Inverse Problems, Inverse Methods and Computerized
              Inversion of Data},
    VOLUME = {24},
      YEAR = {2008},
    NUMBER = {1},
     PAGES = {015009, 5},
      ISSN = {0266-5611,1361-6420},
   MRCLASS = {65J22 (47A52)},
  MRNUMBER = {2384768},
MRREVIEWER = {Hanna\ K.\ Pikkarainen},
       DOI = {10.1088/0266-5611/24/1/015009},
       URL = {https://doi.org/10.1088/0266-5611/24/1/015009},
}

@book{Friedman,
    AUTHOR = {Friedman, Avner},
     TITLE = {Partial differential equations of parabolic type},
 PUBLISHER = {Prentice-Hall, Inc., Englewood Cliffs, NJ},
      YEAR = {1964},
     PAGES = {xiv+347},
   MRCLASS = {35.00 (35.62)},
  MRNUMBER = {181836},
MRREVIEWER = {B.\ Frank\ Jones, Jr.},
}

@article {Flemming_Hofmann,
	AUTHOR = {Flemming, Jens and Hofmann, Bernd},
	TITLE = {Convergence rates in constrained {T}ikhonov regularization:
	equivalence of projected source conditions and variational
	inequalities},
	JOURNAL = {Inverse Problems},
	FJOURNAL = {Inverse Problems. An International Journal on the Theory and
	Practice of Inverse Problems, Inverse Methods and Computerized
	Inversion of Data},
	VOLUME = {27},
	YEAR = {2011},
	NUMBER = {8},
	PAGES = {085001, 11},
	ISSN = {0266-5611,1361-6420},
	MRCLASS = {65J20 (47A52 49J40 49N45 65K15)},
	MRNUMBER = {2819943},
	MRREVIEWER = {Uno\ H\"amarik},
	DOI = {10.1088/0266-5611/27/8/085001},
	URL = {https://doi.org/10.1088/0266-5611/27/8/085001},
}

@article {Hanke_Neubauer_Scherzer,
    AUTHOR = {Hanke, Martin and Neubauer, Andreas and Scherzer, Otmar},
     TITLE = {A convergence analysis of the {L}andweber iteration for
              nonlinear ill-posed problems},
   JOURNAL = {Numer. Math.},
  FJOURNAL = {Numerische Mathematik},
    VOLUME = {72},
      YEAR = {1995},
    NUMBER = {1},
     PAGES = {21--37},
      ISSN = {0029-599X,0945-3245},
   MRCLASS = {65J20 (47H17)},
  MRNUMBER = {1359706},
MRREVIEWER = {Hans-J\"urgen\ Reinhardt},
       DOI = {10.1007/s002110050158},
       URL = {https://doi.org/10.1007/s002110050158},
}

\end{document}